\documentclass[reqno,a4paper,twoside]{amsart}
\usepackage{enumitem}
\setenumerate{label=\textnormal{(\arabic*)}}

\allowdisplaybreaks[3]

\usepackage{amsmath,amssymb,dsfont,mathrsfs,verbatim,bm,geometry,MnSymbol}
\usepackage{mathtools}
\usepackage[latin1]{inputenc}
\usepackage{array,hhline}
\usepackage[raggedright]{titlesec}
\usepackage{tikz}
\usetikzlibrary{matrix}
\usepackage{tikz-cd}
\usepackage{float,subfig}
\usepackage{amsrefs}
\usepackage[final]{microtype}

\titleformat{\section}
{\normalfont\Large\bfseries\center}{\thesection}{1em}{}
\titleformat{\subsection}
{\normalfont\large\bfseries}{\thesubsection}{1em}{}
\titleformat{\subsubsection}
{\normalfont\normalsize\bfseries}{\thesubsubsection}{1em}{}
\titleformat{\paragraph}[runin]
{\normalfont\normalsize\bfseries}{\theparagraph}{1em}{}
\titleformat{\subparagraph}[runin]
{\normalfont\normalsize\bfseries}{\thesubparagraph}{1em}{}

\titlespacing*{\chapter} {0pt}{50pt}{40pt}
\titlespacing*{\section} {0pt}{3.5ex plus 1ex minus .2ex}{2.3ex plus .2ex}
\titlespacing*{\subsection} {0pt}{3.25ex plus 1ex minus .2ex}{1.5ex plus .2ex}
\titlespacing*{\subsubsection}{0pt}{3.25ex plus 1ex minus .2ex}{1.5ex plus .2ex}
\titlespacing*{\paragraph} {0pt}{3.25ex plus 1ex minus .2ex}{1em}
\titlespacing*{\subparagraph} {\parindent}{3.25ex plus 1ex minus .2ex}{1em}

\newtheorem{theorem}{Theorem}[section]
\newtheorem{lemma}[theorem]{Lemma}
\newtheorem{proposition}[theorem]{Proposition}
\newtheorem{corollary}[theorem]{Corollary}

\theoremstyle{definition}

\newtheorem{notations}[theorem]{Notations}

\theoremstyle{remark}
\newtheorem{remark}[theorem]{Remark}

\DeclareMathOperator{\ide}{id}

\DeclareMathOperator{\J}{J}

\DeclareMathOperator{\ima}{Im}
\DeclareMathOperator{\BC}{BC}
\DeclareMathOperator{\BN}{BN}
\DeclareMathOperator{\BP}{BP}
\DeclareMathOperator{\BW}{BW}
\DeclareMathOperator{\Ho}{H}
\DeclareMathOperator{\HH}{HH}
\DeclareMathOperator{\HC}{HC}
\DeclareMathOperator{\HP}{HP}
\DeclareMathOperator{\HN}{HN}

\DeclareMathOperator{\Tot}{Tot}

\newcommand{\ov}{\overline}
\newcommand{\ot}{\otimes}

\newcommand{\wt}{\widetilde}
\newcommand{\ep}{\epsilon}

\newcommand{\de}{\delta}

\newcommand{\bx}{\mathbf x}
\newcommand{\bc}{\mathbf c}

\newcommand{\byy}{\mathbf y}
\newcommand{\bz}{\mathbf z}
\newcommand{\hs}{\hspace{-0.8pt}}

\newcommand{\xcirc}{\hs\circ\hs}

\newcommand{\sR}{\scriptstyle R \textstyle}
\newcommand{\sdR}{\scriptstyle R \displaystyle}

\newcommand{\sw}{\mathrm{sw}}

\DeclareMathAlphabet{\mathpzc}{OT1}{pzc}{m}{it}

\begin{document}

\title{Cyclic homology of cleft extensions of algebras}

\author{Jorge A. Guccione}
\address{Departamento de Matem\'atica\\ Facultad de Ciencias Exactas y Naturales-UBA, Pabell\'on~1-Ciudad Universitaria\\ Intendente Guiraldes 2160 (C1428EGA) Buenos Aires, Argentina.}
\address{Instituto de Investigaciones Matem\'aticas ``Luis A. Santal\'o"\\ Facultad de Ciencias Exactas y Natu\-ra\-les-UBA, Pabell\'on~1-Ciudad Universitaria\\ Intendente Guiraldes 2160 (C1428EGA) Buenos Aires, Argentina.}
\email{vander@dm.uba.ar}

\author{Juan J. Guccione}
\address{Departamento de Matem\'atica\\ Facultad de Ciencias Exactas y Naturales-UBA\\ Pabell\'on~1-Ciudad Universitaria\\ Intendente Guiraldes 2160 (C1428EGA) Buenos Aires, Argentina.}
\address{Instituto Argentino de Matem\'atica-CONICET\\ Savedra 15 3er piso\\ (C1083ACA) Buenos Aires, Argentina.}
\email{jjgucci@dm.uba.ar}

\thanks{Jorge A. Guccione and Juan J. Guccione were supported by UBACyT 20020110100048 (UBA) and PIP 11220110100800CO (CONICET)}

\author[C. Valqui]{Christian Valqui}
\address{Pontificia Universidad Cat\'olica del Per\'u - Instituto de Matem\'atica y Ciencias Afi\-nes, Secci\'on Matem\'aticas, PUCP, Av. Universitaria 1801, San Miguel, Lima 32, Per\'u.}
\email{cvalqui@pucp.edu.pe}

\thanks{Christian Valqui was supported by PUCP-DGI-2013-3036.}


\begin{abstract} Let $k$ be a commutative algebra with $\mathds{Q}\subseteq k$ and let $(E,p,i)$ be a cleft extension of $A$. We obtain a new mixed complex, simpler than the canonical one, giving the Hochschild and cyclic homologies of $E$ relative to $\ker(p)$. This complex resembles the canonical reduced mixed complex of an augmented algebra. We begin the study of our complex showing that it has a harmonic decomposition like to the one considered by Cuntz and Quillen for the normalized mixed complex of an algebra.
\end{abstract}

\maketitle

\section*{Introduction}
Let $k$ be a commutative ring such that $\mathds{Q}\subseteq k$, and let $A$ be an associative unital $k$-algebra. A {\em cleft extension} of $A$ is a triple $(E,p,i)$, consisting of a $k$-algebra~$E$ and $k$-algebra morphisms $i\colon A\to E$ and $p\colon E\to A$, such that $p \circ i=\ide_A$. When applications $i$ and $p$ are evident we simply say that $E$ is a cleft extension of $A$ instead of saying that $(E,p,i)$ it is. A {\em morphism} $f\colon (E,p,i) \longrightarrow (F,q,j)$, of a cleft extension of an algebra $A$ to a cleft extension of an algebra $B$, is a morphisms of algebras $f\colon E\to F$ such that there exists a (necessarily unique) morphism $\varphi\colon A\to B$ satisfying $f\circ i=j\circ \varphi$ and $q\circ f=\varphi \circ p$.

For example, given a graded algebra $E = \bigoplus_{i\ge 0} A_i$, each subalgebra of $E$ including $A_0$ is a cleft extension of $A_0$, and the same is true for each quotient $E/I$, of $E$  by a two sided ideal $I\subseteq A_+:= \bigoplus_{i>0} A_i$. This includes a lot of well known examples of constructive nature, such as tensor algebras, truncated tensor algebras, truncated quiver algebras,  symmetric algebras, exterior algebras, general quadratic algebras, monomial algebras, Rees algebras, etcetera. Let $R$ be a commutative ring such that $k\subseteq R$. As was point out in \cite{B}, also there are examples appearing for structural reasons. For instance:

\begin{enumerate}

\smallskip

\item An $R$-algebra $E$ is a cleft extension of $A:=E/\J(E)$, where $\J(E)$ denotes the Jacobson radical of $E$, when:

\begin{enumerate}

\smallskip

\item[(i)] $E$ is a basic semiperfect algebra or $A$ is $R$-projective, $\J(E)$ is nilpotent and the Hochschild dimension of $A$ is lesser or equal than $1$ (see \cite{Pierce} and \cite{Rotman}),

\smallskip

\item[(ii)] $R$ is a field and $\dim_R(E) <\infty$ (see \cite{Eilenberg} and \cite{MacLane}).

\smallskip

\end{enumerate}

\item If $R$ is a field, then, by the Levi-Malcev Theorem (see \cite{Abe}), the universal enveloping algebra $U(L)$, of a finite-dimensional Lie algebra $L$, is a cleft extension of $U(L/J)$, where $J$ is the radical of $L$.

\smallskip

\item If $A$ is quasi-free, then any nilpotent extension of $A$ is cleft (see \cite{C-Q1}).

\smallskip

\end{enumerate}
A special interesting type of cleft extensions are the split-by-nilpotent extension algebras (see \cite{Assem}, \cite{Kha} and \cite{LMZ}).

\smallskip

There is a canonical way to construct a cleft extension of an algebra $A$. Given a an $A$-bimodule $M$ endowed with an associative $A$-bimodule map
$$
\begin{tikzpicture}
\draw [->] (0.5,0)node[left=0pt]{$M\ot_A M$}  -- (2,0) node[right=0pt]{$M,$};
\draw[|->] (0.5,-0.4) node[left=0pt]{$m\ot m'$} -- (1.7,-0.4) node[right=0pt]{$m \smalltriangledown m'$};
\end{tikzpicture}
$$
the triple $(A\ltimes_{\!\smalltriangledown} M, i, p)$, where $A\ltimes_{\!\smalltriangledown} M$ is the algebra with underlying vector space $A\times M$ and multiplication map
\[
(a,m)(a',m') := (aa',am' + ma' + m \smalltriangledown m'),
\]
$i\colon A\longrightarrow A\ltimes_{\!\smalltriangledown} M$ is the map defined by $i(a):=(a,0)$ and $p \colon A\ltimes_{\!\smalltriangledown} M \longrightarrow A$ is the map defined by $p(a,m):=a$, is a cleft extension of $A$, that is called {\em the~$\smalltriangledown$-extension} of $A$ or the {\em cleft extension of $A$ associated with~$\smalltriangledown$}. Moreover it is not difficult to see that each clef extension is isomorphic to one of this type.

\smallskip

Let $k$ be a commutative ring and let $C$ be a $k$-algebra. If $K$ is a subalgebra of $C$, we will say that $C$ is a $K$-algebra. In this paper we begin the study of the Hochschild, cyclic, negative and periodic homologies of cleft extensions $(E,p,i)$ of a $K$-algebra~$A$. It is easy to see that these homologies are the direct sum of the corresponding homologies of the $K$-algebra $A$ and the corresponding homologies of the $K$-algebra $E$, relative to $\ker(p)$. So we restrict our attention to the last ones. Moreover, by the previous discussion, we can assume that $E$ is a cleft extension $A\ltimes_{\!\smalltriangledown} M$ associated with an associative $A$-bimodule map $\smalltriangledown$. Our main result is Theorem~\ref{th3.2}, in which we obtain a double mixed complex $(\hat{X},\hat{b},\hat{d},\hat{B})$, given these homologies, whose associated mixed complex $(\breve{X},\breve{b},\breve{B})$ is simpler than the canonical mixed complex of $A\ltimes_{\!\smalltriangledown} M$ relative to $M$. We also introduce an begin the study of an harmonic decomposition of this complex. In a forthcoming paper we are going to obtain an still simpler complex under the hypothesis that $M$ is isomorphic to $A\ot_k V$ as a left $A$-module. We hope that this allows us to get explicit computations.

\smallskip

The paper is organized in the following way:

\smallskip

In Section~\ref{preliminares} we recall some well known definitions and results. Among them, the perturbation lemma, which we will use again and again in the rest of the paper, and the definition of double mixed complex, which we got from \cite{Co}.

\smallskip

Section~\ref{relative} is devoted to establishing the main results in this paper. In fact, $(\hat{X},\hat{b},\hat{d},\hat{B})$ can be thought as a double mixed complex associated to the $3$-tuple $(A,M,\smalltriangledown)$, and this association is functorial in an evident sense. For $0\le 2w\le v$, let $X_{vw}$ be the direct sum of all the tensor products $X_0\ot \cdots \ot X_n$ such that $X_0=M$, $X_i=M$ for $w$ indices $i>0$ and $X_i=\ov{A}$ for the other ones, where $n=v+w$ and $\ov{A} = A/k$. Let $b\colon X_{vw}\longrightarrow X_{v-1,w}$ be the map given by the same formula as the Hochschild boundary map of an algebra, where the meaning of the concatenation $x_ix_{i+1}$ of two consecutive factors in a simple tensor is the one given in item~(3) of Notation~\ref{not1.2}. Let $t\colon X_{vw}\longrightarrow X_{vw}$ be
the map defined by
$$
t(x_0\ot\cdots\ot x_n) = (-1)^{in}x_i\ot \cdots \ot x_n\ot x_0\ot\cdots\ot x_{i-1},
$$
where $i$ denotes the last index such that $x_i\in M$ and let $N = \ide + t + t^2 + \cdots + t^w$. The double mixed complex $(\hat{X},\hat{b},\hat{d},\hat{B})$ has objects $\hat{X}_{vw} = X_{vw}\oplus X_{v-1,w}$. The boundary maps are given by
\[
\hat{b}(\bx,\byy) = \bigl(b(\bx) + (\ide-t) (\byy),-b(\byy)\bigr)\qquad\text{and}\qquad \hat{d}(\bx,\byy) = \bigl(d(\bx),-d'(\byy)\bigr),
\]
where $d,d'\colon X_{vw}\longrightarrow X_{v,w-1}$ are maps depending on the map $\smalltriangledown$, and the Connes operator is given by $\hat{B}(\bx,\byy) = (0,N(\bx))$. So, it resembles the reduced mixed complex of an augmented algebra. Since the maps $t$ and $N$ satisfy
\begin{equation}
\ima(1-t) = \ker(N)\qquad\text{and}\qquad \ima(N)=\ker(1-t),\label{eq1}
\end{equation}
the cyclic homology of $E$ relative to $M$ is the homology of the quotient complex of $(X,b,d)$ by the image of $\ide-t$. Indeed, this also follows from the fact that $(\hat{X},\hat{b},\hat{d},\hat{B})$ satisfies the Connes property (\cite{C-Q2}), which is another consequence of the equalities~\eqref{eq1}.
\smallskip

Our purpose in Section~\ref{The harmonic decomposition} is to show that $(\breve{X},\breve{b},\breve{B})$ has a harmonic decomposition like the one studied in \cite{C-Q2}. In order to carry out this task we need to define a de Rham coboundary map and a Karoubi operator on $(\breve{X},\breve{b})$. Actually it will be convenient for us to work with a new double mixed complex, namely $(\ddot{X},\ddot{d},\ddot{b},\ddot{B})$, whose associated mixed complex is also $(\breve{X},\breve{b},\breve{B})$. As in \cite{C-Q2} the Karoubi operator $\ddot{\kappa}$ of $(\ddot{X},\ddot{d},\ddot{b})$ commutes with $\ddot{b}$ and $\ddot{d}$ and satisfies a polynomial equation $P_w(\kappa)$ on each $\ddot{X}_{vw}$. Thus we have the harmonic decomposition $\ddot{X} = P(\ddot{X})\oplus P^{\perp}(\ddot{X})$, where $P$ is the spectral projection onto the generalized nullspace for $\ide - \ddot{\kappa}$ and $P^{\perp} = 1-P$. The first component of this decomposition is $\ddot{B}$-acyclic and the second one is $\ddot{d}$-acyclic and killed by $\ddot{B}$. Hence $(\ddot{X},\ddot{d},\ddot{b})$ has the Connes property. We finish the section by giving an explicit description of $P(\ddot{X})$.

\smallskip

Next we introduce some notations that we will use throughout this paper.

\begin{notations}\label{not1.1} Let $V$ be a $K$-bimodule and let $M$ be a $C$-bimodule.

\begin{enumerate}

\smallskip

\item We set $\ov{C}:= C/K$. Moreover, given $c\in C$, we also denote by $c$ the class of $c$ in $\ov{C}$.

\smallskip

\item We use the unadorned tensor symbol $\ot$ to denote the tensor product $\ot_{\!K}$.

\smallskip

\item We let $V^{\ot^n}$ denote the $n$-fold power tensor of $V$.

\smallskip

\item Given $c_0,\dots, c_r \in C$ and $i<j$, we write $\mathbf{c}_i^j:= c_i\ot\cdots\ot c_j$.

\smallskip

\item Given a $K$-bimodule $M$, we let $M\ot$ denote the quotient $M/[M,K]$, where $[M,K]$ is the $k$-vector subspace of $M$ generated by all the commutators $m\lambda - \lambda m$, with $m\in M$ and $\lambda\in K$. Moreover, for $m\in M$, we let $[m]$ denote the class of $m$ in $M\ot$.

\end{enumerate}
\end{notations}

\begin{notations}\label{not1.2} Let $E := A\ltimes_{\!\smalltriangledown} M$ be a cleft extension.

\begin{enumerate}

\smallskip

\item We let
$$
\pi_{\!A}\colon E\longrightarrow A\qquad\text{and}\qquad \pi_{\!M}\colon E \longrightarrow M
$$
denote the maps defined by $\pi_{\!A}(a,m):=a$ and $\pi_{\!M}(a,m):=m$, respectively.

\smallskip

\item Given $x,y\in A\bigcup M$ we set
\[
x \smalltriangledown y := \begin{cases} x \smalltriangledown y &\text{if $x,y\in M$,}\\0 &\text{otherwise.}\end{cases}
\]
We extend this definition for $x,y\in E$ by linearity.

\smallskip

\item Given $x,y\in A\bigcup M$ we set
\[
xy := \begin{cases} \text{the product of $x$ and $y$ in $A$} &\text{if $x,y\in A$,}\\ \text{the left action of $x$ on $y$} &\text{if $x\in A$ and $y\in M$,}\\ \text{the right action of $y$ on $x$} &\text{if $x\in M$ and $y\in A$,}\\ 0 &\text{if $x,y\in M$.}
\end{cases}
\]
We extend this definition for $x,y\in E$ by linearity.

\smallskip

\item For $0\le w\le n$, let $B^n_w\subseteq \ov{E}^{\ot^n}$ be the $k$-submodule spanned by the simple tensors $\bx_1^n$ such that exactly $w$ of the $x_i$'s belong to $M$, while the other ones belong to $\ov{A}$. To unify expressions we make the convention that $B_0^0 := k$ and $B^n_w := 0$, for $w<0$ or $n<w$.

\smallskip

\item We will say that a simple tensor $\bx_0^n\in E\ot\ov{E}^{\ot^n}\ot$ is very simple if $x_i\in A\cup M$ for all $i$.

\smallskip

\item For $\bx_0^n\in \bigl(M\ot B^n_w\ot\bigr)\bigcup \bigl(A\ot B^n_{w+1}\ot\bigr)$ and $0\le l\le n$, we define $\mu_j(\bx_0^n)$ by
$$
\qquad\quad \mu_j(\bx_0^n) := \begin{cases} (-1)^j \bx_0^{j-1}\ot x_jx_{j+1}\ot \bx_{j+2}^n &\text{if $0\le j< n$,}\\ (-1)^n x_nx_0\ot\bx_1^{n-1}&\text{If $j = n$.}\end{cases}
$$
Moreover we set
\begin{alignat*}{2}
&\qquad\quad\mu_0^A(\bx_0^n):= \pi_{\!A}(x_0x_1)\ot\bx_2^n,\qquad &&\mu_n^A(\bx_0^n) = (-1)^n \pi_{\!A}(x_nx_0)\ot \bx_1^{n-1},\\
&\qquad\quad\mu_0^M(\bx_0^n):= \pi_{\!M}(x_0x_1)\ot\bx_2^n,\qquad &&\mu_n^M(\bx_0^n) = (-1)^n \pi_{\!M}(x_nx_0)\ot \bx_1^{n-1}.
\end{alignat*}

\smallskip

\item For $\bx_0^n\in E\ot\ov{E}^{\ot^n}\ot$ and $0\le j\le n$, we define $\varrho_j(\bx_0^n)$ by
$$
\varrho_j(\bx_0^n) = \begin{cases} (-1)^j \bx_0^{j-1}\ot x_j \smalltriangledown x_{j+1}\ot \bx_{j+2}^n &\text{if $0\le j< n$,}\\ (-1)^n x_n \smalltriangledown x_0\ot\bx_1^{n-1}&\text{If $j = n$.}\end{cases}
$$

\smallskip

\item Let $\bx_0^n \in E\ot\ov{E}^{\ot^n}\ot$ be a very simple tensor. We let $i(\bx_0^n)$ and denote the last index $i$ such that $x_i\in M$.

\smallskip

\item For a very simple tensor $\bx_0^n\in E\ot\ov{E}^{\ot^n}\ot$, we define
$$
t(\bx_0^n):= (-1)^{i(\bx_0^n)n} \bx_{i(\bx_0^n)}^n\ot \bx_0^{i(\bx_0^n)-1}.
$$
We extend this definition to $E\ot\ov{E}^{\ot^n}\ot$ by linearity.
\end{enumerate}

\end{notations}

\section{Preliminaries}\label{preliminares}
In this section we recall some well known definitions and results that we will use in the rest of the paper. Let $C$ be a $k$-algebra.
\subsection{Double and triple complexes}
A double complex $\mathcal{X} := (X,d^v,d^h)$ of $C$-modules, is a family $(X_{pq})_{p,q\in\mathds{Z}}$ of $C$-modules, together with $C$-linear maps
\[
d^h\colon X_{pq}\longrightarrow X_{p-1,q}\qquad\text{and}\qquad d^v\colon X_{pq}\longrightarrow  X_{p,q-1},
\]
such that $d^h\xcirc d^h=0$, $d^v\xcirc d^v=0$ and $d^v \xcirc d^h + d^h\xcirc d^v =0$. The total complex of $(X,d^v,d^h)$ is the complex $\Tot(\mathcal{X}) = (X,d)$, in which
\[
X_n := \prod_p X_{p,n-p}\qquad\text{and}\qquad d := d^v+d^h.
\]
A morphism of double complexes $f\colon (X,d^v,d^h)\longrightarrow (Y,\de^v,\de^h)$ is a family of maps $f\colon X_{pq}\longrightarrow Y_{pq}$, such that $\de^v\xcirc f = f\xcirc d^v$ and $\de^h \xcirc f= f\xcirc d^h$. The morphism from $\Tot(X,d^v,d^h)$ to $\Tot(Y,\de^v,\de^h)$ induced by $f$ will be denoted $\Tot(f)$.

\smallskip

Similarly, one can give the notions of triple complex $\mathcal{X} := (X,d^v,d^h,d^d)$ and of morphism of triple complexes. For a triple complex $\mathcal{X}$, there are three ways for constructing a double complex by taking total complexes of double complexes. We call each one of these double complexes a partial total complex of $\mathcal{X}$. The total complex $\Tot(\mathcal{X})$ of $\mathcal{X}$, is the total complex of any of its partial total complexes. Of course, $\Tot(\mathcal{X})$ is independently of the chosen way to construct it.

\subsection{Mixed complexes}
In this subsection we recall briefly the notion of mixed complex. For more details about this concept we refer to~\cite{B} and~\cite{Ka1}.

\smallskip

A {\em mixed complex} $(X,b,B)$ is a graded $k$-module $(X_n)_{n\ge 0}$, endowed with morphisms
$$
b\colon X_n\longrightarrow X_{n-1}\qquad\text{and}\qquad B\colon X_n\longrightarrow X_{n+1},
$$
such that
$$
b\xcirc b = 0,\quad B\xcirc B = 0\quad\text{and}\quad B\xcirc b + b\xcirc B = 0.
$$
A {\em morphism of mixed complexes} $f\colon (X,b,B)\longrightarrow (Y,d,D)$ is a family of maps $f\colon X_n\longrightarrow Y_n$, such that $d\xcirc f = f\xcirc b$ and $D\xcirc f= f\xcirc B$. Let $u$ be a degree~$2$ variable. A mixed complex $\mathcal{X}:= (X,b,B)$ determines a double complex
$$
\begin{tikzpicture}
\begin{scope}[yshift=-0.47cm,xshift=-6cm]
\draw (0.5,0.5) node {$\BP(\mathcal{X})=$};
\end{scope}
\begin{scope}[scale=0.8]
\matrix(BPcomplex) [matrix of math nodes,row sep=2.5em, text height=1.5ex, text
depth=0.25ex, column sep=2.5em, inner sep=0pt, minimum height=5mm,minimum width =9.5mm]
{&\vdots &\vdots &\vdots &\vdots\\
\cdots & X_3 u^{-1} & X_2 u^0 & X_1 u^{} & X_0 u^2\\
\cdots & X_2 u^{-1} & X_1 u^0 & X_0 u\\
\cdots & X_1 u^{-1} & X_0 u^0\\
\cdots & X_0 u^{-1}\\ };
\draw[->] (BPcomplex-1-2) -- node[right=1pt,font=\scriptsize] {$b$} (BPcomplex-2-2);
\draw[->] (BPcomplex-1-3) -- node[right=1pt,font=\scriptsize] {$b$} (BPcomplex-2-3);
\draw[->] (BPcomplex-1-4) -- node[right=1pt, font=\scriptsize] {$b$} (BPcomplex-2-4);
\draw[->] (BPcomplex-1-5) -- node[right=1pt, font=\scriptsize] {$b$} (BPcomplex-2-5);
\draw[<-] (BPcomplex-2-1) -- node[above=1pt,font=\scriptsize] {$B$} (BPcomplex-2-2);
\draw[<-] (BPcomplex-2-2) -- node[above=1pt,font=\scriptsize] {$B$} (BPcomplex-2-3);
\draw[<-] (BPcomplex-2-3) -- node[above=1pt,font=\scriptsize] {$B$} (BPcomplex-2-4);
\draw[<-] (BPcomplex-2-4) -- node[above=1pt,font=\scriptsize] {$B$} (BPcomplex-2-5);
\draw[->] (BPcomplex-2-2) -- node[right=1pt,font=\scriptsize] {$b$} (BPcomplex-3-2);
\draw[->] (BPcomplex-2-3) -- node[right=1pt,font=\scriptsize] {$b$} (BPcomplex-3-3);
\draw[->] (BPcomplex-2-4) -- node[right=1pt, font=\scriptsize] {$b$} (BPcomplex-3-4);
\draw[<-] (BPcomplex-3-1) -- node[above=1pt,font=\scriptsize] {$B$} (BPcomplex-3-2);
\draw[<-] (BPcomplex-3-2) -- node[above=1pt,font=\scriptsize] {$B$} (BPcomplex-3-3);
\draw[<-] (BPcomplex-3-3) -- node[above=1pt,font=\scriptsize] {$B$} (BPcomplex-3-4);
\draw[->] (BPcomplex-3-2) -- node[right=1pt,font=\scriptsize] {$b$} (BPcomplex-4-2);
\draw[->] (BPcomplex-3-3) -- node[right=1pt,font=\scriptsize] {$b$} (BPcomplex-4-3);
\draw[<-] (BPcomplex-4-1) -- node[above=1pt,font=\scriptsize] {$B$} (BPcomplex-4-2);
\draw[<-] (BPcomplex-4-2) -- node[above=1pt,font=\scriptsize] {$B$} (BPcomplex-4-3);
\draw[->] (BPcomplex-4-2) -- node[right=1pt,font=\scriptsize] {$b$} (BPcomplex-5-2);
\draw[<-] (BPcomplex-5-1) -- node[above=1pt,font=\scriptsize] {$B$} (BPcomplex-5-2);
\end{scope}
\end{tikzpicture}
$$
\noindent where $b(\bx u^i):= b(\bx)u^i$ and $B(\bx u^i):= B(\bx)u^{i-1}$. By deleting the positively numbered columns we obtain a subcomplex $\BN(\mathcal{X})$ of $\BP(\mathcal{X})$. Let $\BN'(\mathcal{X})$ be the kernel of the canonical surjection from $\BN(\mathcal{X})$ to $(X,b)$. The quotient double complex $\BP(\mathcal{X})/\BN'(\mathcal{X})$ is denoted by $\BC(\mathcal{X})$. The homology groups $\HC_*(\mathcal{X})$, $\HN_*(\mathcal{X})$ and $\HP_*(\mathcal{X})$, of the total complexes of $\BC(\mathcal{X})$, $\BN(\mathcal{X})$ and $\BP(\mathcal{X})$ respectively, are called the {\em cyclic}, {\em negative} and {\em periodic homology groups} of $\mathcal{X}$. The homology $\HH_*(\mathcal{X})$, of $(X,b)$, is called the {\em Hochschild homology} of $\mathcal{X}$. Finally, it is clear that a morphism $f\colon\mathcal{X}\to\mathcal{Y}$ of mixed complexes induces a morphism from the double complex $\BP(\mathcal{X})$ to the double complex $\BP(\mathcal{Y})$.

\smallskip

Following \cite{Co} by a double mixed complex we will understand a bigraded module $X$ equipped with three $k$-linear maps of degree $\pm 1$: $b$, that lowers the first index and fixes the second one, $\beta$, that fixes the first index and lowers the second one, and $B$, that fixes the first index and increases the second one. These maps satisfy
\[
0 = b^2 = \beta^2 = B^2 = \beta\xcirc b + b\xcirc \beta = \beta \xcirc B + B\xcirc \beta = b \xcirc B + B\xcirc b.
\]
The {\em total mixed complex} of a double mixed complex $(X,b,\beta,B)$ is the mixed complex $(X,b + \beta,B)$, where $(X,b+\beta):= \Tot(X,b,\beta)$ and $B_n:= \bigoplus_{i+j=n} B_{ij}$. By definition, the Hochschild, cyclic, periodic and negative homologies of $(X,b,\beta,B)$ are the Hochschild, cyclic, periodic and negative homologies of $(X,b+\beta,B)$, respectively.

A morphism of double mixed complexes $f\colon (X,b,\beta,B)\longrightarrow (Y,d,\delta,D)$ is a family $f\colon X_{ij}\longrightarrow Y_{ij}$, such that $d\xcirc f = f\xcirc b$, $\delta\xcirc f = f\xcirc \beta$ and $D\xcirc f= f\xcirc B$. It is obvious that the correspondence $(X,b,\beta,B)\mapsto (X,b+\beta,B)$ is functorial.

\subsection{The relative Hochschild and cyclic homologies}
Let $C$ be a $K$-algebra. By definition, the {\em normalized mixed complex of the $K$-algebra $C$} is the mixed complex $\cramped{(C\ot \ov{C}^{\ot^*}\ot,b,B)}$, where $b$ is the canonical Hochschild boundary map and the Connes operator $B$ is given by
$$
B([\bc_{0r}]):= \sum_{i=0}^r [(-1)^{ir} \ot \bc_{ir}\ot \bc_{0,i-1}].
$$
The {\em cyclic}, {\em negative}, {\em periodic} and {\em Hochschild homology groups} $\HC^K_*(C)$, $\HN^K_*(C)$, $\HP^K_*(C)$ and $\HH^K_*(C)$ of $C$ are the respective homology groups of $\cramped{(C\ot\ov{C}^{\ot^*}\ot,b,B)}$.

\smallskip

Let $I$ be a two sided ideal of $C$ and let $D:= C/I$. The {\em cyclic, negative, periodic and Hochschild homologies $\HC^K_*(C,I)$, $\HN^K_*(C,I)$, $\HP^K_*(C,I)$ and $\HH^K_*(C,I)$, of the $K$-algebra $C$ relative to~$I$}, are by definition the respective homologies of the mixed complex
$$
\ker \bigl(\!\!\!\begin{tikzcd}[row sep=5.4em, column sep=5.4em]
(C\ot \ov{C}^{\ot^*}\ot,b,B)  \arrow{r}{\pi} & (D\ot \ov{D}^{\ot^*}\ot,b,B)
\end{tikzcd}\!\!\!\bigr),
$$
where $\pi$ is the map induced by the canonical projection from $C$ onto $D$.

\subsection{The perturbation lemma}
Next, we recall the perturbation lemma. We present the version given in~\cite{C}.

\smallskip

A {\em homotopy equivalence data}
$$
\begin{tikzpicture}[label distance=5mm]
\draw (0,0) node[label=left:(b)] {$(Y,\partial)$};\draw (2.5,0) node {$(X,d)$};\draw[<-] (0.6,0.12) -- node[above=-2pt,font=\scriptsize] {$p$}(1.9,0.12);\draw[->] (0.6,-0.12) -- node[below=-2pt,font=\scriptsize] {$i$}(1.9,-0.12);\draw (4,0) node {$X_*$};\draw (6.5,0) node {$X_{*+1}$};\draw[->] (4.4,0) -- node[above=-2pt,font=\scriptsize] {$h$}(5.9,0);
\end{tikzpicture}
$$
consists of the following:

\begin{enumerate}

\smallskip

\item Chain complexes $(Y,\partial)$, $(X,d)$ and quasi-isomorphisms $i$, $p$ between them,

\smallskip

\item A homotopy $h$ from $i\xcirc p$ to $\ide$.
\end{enumerate}

\smallskip

A {\em perturbation} of (b) is a map $\delta\colon X_*\longrightarrow X_{*-1}$ such that $(d+\delta)^2 = 0$. We call it {\em small} if $\ide -\delta\xcirc h$ is invertible. In this case we write $A = (\ide -\delta\xcirc h)^{-1}\xcirc\delta$ and we consider de diagram
$$
\begin{tikzpicture}[label distance=5mm]
\draw (0,0) node[label=left:(c)] {$(Y,\partial^1)$};\draw (2.5,0) node {$(X,d)$};\draw[<-] (0.6,0.12) -- node[above=-2pt,font=\scriptsize] {$p^1$}(1.9,0.12);\draw[->] (0.6,-0.12) -- node[below=-2pt,font=\scriptsize] {$i^1$}(1.9,-0.12);\draw (4,0) node {$X_*$};\draw (6.5,0) node {$X_{*+1}$};\draw[->] (4.4,0) -- node[above=-2pt,font=\scriptsize] {$h^1$}(5.9,0);
\end{tikzpicture}
$$
where
$$
\partial^1:=\partial + p\xcirc A\xcirc i,\quad i^1:= i + h\xcirc A\xcirc i,\quad p^1:= p + p\xcirc A\xcirc h,\quad h^1:= h + h\xcirc A\xcirc h.
$$
A {\em deformation retract} is a homotopy equivalence data such that $p\xcirc i =\ide$. A deformation retract is called {\em special} if $h\xcirc i = 0$, $p\xcirc h = 0$ and $h\xcirc h = 0$.

\smallskip

In all the cases considered in this paper the morphism $\delta\xcirc h$ is locally nilpotent. Consequently, $(\ide -\delta\xcirc h)^{-1} =\sum_{n=0}^{\infty} (\delta\xcirc h)^n$.

\begin{theorem}[{\cite{C}}]\label{lema de perturbacion} If $\delta$ is a small perturbation of~\emph{(b)}, then the diagram~\emph{(c)} is an homotopy equivalence data. Furthermore, if~\emph{(b)} is a special deformation retract, then so it is~\emph{(c)}.
\end{theorem}

\subsection{The suspension} The {\em suspension of a chain complex} $(X,b)$ is the complex $(X,b)[1] := (X[1],b[1])$, defined by
$$
X[1]_* := X_{*-1}\qquad\text{and}\qquad b[1]_* := -b_{*-1}.
$$
Similarly, the {\em suspension of a chain double complex} $(X,b,\beta)$ is the complex
$$
(X,b,\beta)[0,1] := (X[0,1],b[0,1],\beta[0,1]),
$$
defined by $X[0,1]_{**} := X_{*,*-1}$, $b[0,1]_{**} := -b_{*,*-1}$ and $\beta[0,1]_{**} := -\beta_{*,*-1}$.

\section{The relative cyclic homology of a cleft extension}\label{relative}
Let $C$ be a $k$-algebra and let $ p \colon C\to A$ be a morphism of $k$-algebras. Assume there exists a $k$-algebra morphism $i\colon A\to C$ such that $p \xcirc i = \ide$. Then $M := \ker(p)$ is naturally an $A$-bimodule endowed with an associative multiplication $\smalltriangledown\colon M\ot_A M  \longrightarrow M$ and $C$ is isomorphic to the cleft extension $E := A\ltimes_{\! \smalltriangledown} M$. Let $K$ be a $k$-subalgebra of $A$ and let
$$
\begin{tikzcd}[column sep=1cm]
(\breve{\mathfrak{X}},\breve{\mathfrak{b}},\breve{\mathfrak{B}}) :=\ker\bigl((E\ot \ov{E}^{\ot^*}\ot,b,B)  \arrow{r}{\pi} & (A\ot\ov{A}^{\ot^*}\ot,b,B)\bigr),
\end{tikzcd}
$$
\noindent where $\pi$ is the morphism induced by the canonical surjection of $E$ on $A$. Using that $E$ is a cleft extension of $A$, it is easy to see that the short exact sequence
$$
\begin{tikzcd}[column sep=1cm]
0 \arrow{r} & (\breve{\mathfrak{X}},\breve{\mathfrak{b}},\breve{\mathfrak{B}}) \arrow{r} & {(E\ot \ov{E}^{\ot^*}\ot,b,B)} \arrow{r}{\pi} & (A\ot\ov{A}^{\ot^*}\ot,b,B) \arrow{r} & 0
\end{tikzcd}
$$
\noindent splits. Hence,
\begin{align*}
& \HH^K(E) = \HH^K(A)\oplus \HH^K_*(E,M),\\
& \HC^K(E) = \HC^K(A)\oplus \HC^K_*(E,M),\\
& \HP^K(E) = \HP^K(A)\oplus \HP^K_*(E,M),\\
& \HN^K(E) = \HN^K(A)\oplus \HN^K_*(E,M).
\end{align*}
So in order to compute the type cyclic homologies of the $K$-algebra $E$, it suffices to calculate those of the $K$-algebra $A$ and those of the $K$-algebra $E$ relative to $M$. With this in mind, in this section we obtain a double mixed complex, simpler than the canonical one, giving the Hochschild, cyclic, periodic and negative homologies of the $K$-algebra $E$ relative to $M$. Then we show that the cyclic homology of the $K$-algebra $E$ relative to $M$ is also given by a still simpler complex. From now on we often will use indices $v,w$ and $n$, which always will satisfy the relation $n=v+w$.

Let $\hat{\mathfrak{X}}_{vw} := (M\ot B^n_w\ot)\oplus (A\ot B^n_{w+1}\ot)$. It is evident that $(\breve{\mathfrak{X}},\breve{\mathfrak{b}})$ is the total complex of the second quadrant double complex
$$
\begin{tikzcd}[column sep=0.8cm,row sep=0.8cm]
& \vdots  \arrow{d}{\hat{\mathfrak{b}}} & \vdots \arrow{d}{\hat{\mathfrak{b}}}& \vdots \arrow{d}{\hat{\mathfrak{b}}}\\
\dots \arrow{r}{\hat{\mathfrak{d}}} & \hat{\mathfrak{X}}_{22} \arrow{r}{\hat{\mathfrak{d}}} \arrow{d}{\hat{\mathfrak{b}}} & \hat{\mathfrak{X}}_{21} \arrow{r}{\hat{\mathfrak{d}}} \arrow{d}{\hat{\mathfrak{b}}} & \hat{\mathfrak{X}}_{20} \arrow{d}{\hat{\mathfrak{b}}}\\
\dots \arrow{r}{\hat{\mathfrak{d}}} & \hat{\mathfrak{X}}_{12} \arrow{r}{\hat{\mathfrak{d}}} \arrow{d}{\hat{\mathfrak{b}}} & \hat{\mathfrak{X}}_{11} \arrow{r}{\hat{\mathfrak{d}}} \arrow{d}{\hat{\mathfrak{b}}} & \hat{\mathfrak{X}}_{10} \arrow{d}{\hat{\mathfrak{b}}}\\
\dots \arrow{r}{\hat{\mathfrak{d}}} & \hat{\mathfrak{X}}_{02} \arrow{r}{\hat{\mathfrak{d}}} & \hat{\mathfrak{X}}_{01} \arrow{r}{\hat{\mathfrak{d}}} & \hat{\mathfrak{X}}_{00},
\end{tikzcd}
$$
\noindent where the boundary maps
$$
\hat{\mathfrak{b}}\colon \hat{\mathfrak{X}}_{vw} \longrightarrow \hat{\mathfrak{X}}_{v-1,w}\qquad\text{and}\qquad \hat{\mathfrak{d}}\colon \hat{\mathfrak{X}}_{vw} \longrightarrow \hat{\mathfrak{X}}_{v,w-1}
$$
are defined by the formulas
$$
\hat{\mathfrak{b}}(\bx,\byy):= (\mathfrak{b}^0(\bx)+\alpha(\byy),\mathfrak{b}^1(\byy)) \qquad\text{and}\qquad \hat{\mathfrak{d}}(\bx,\byy) := \sum_{j=0}^n (\varrho_j(\bx), \varrho_j(\byy)),
$$
in which
$$
\mathfrak{b}^0:= \sum_{j=0}^n\mu_l,\qquad\mathfrak{b}^1:= \mu_0^A + \sum_{j=1}^{n-1}\mu_l + \mu_n^A\qquad \text{and}\qquad \alpha:= \mu_0^M + \mu_n^M.
$$
Furthermore $(\hat{\mathfrak{X}},\hat{\mathfrak{b}},\hat{\mathfrak{d}},\hat{\mathfrak{B}})$, where $\hat{\mathfrak{B}}$ is defined by
$$
\hat{\mathfrak{B}}(\bx,\byy):= (0,B(\bx)+B(\byy)),
$$
is a double mixed complex and its total mixed complex is $(\breve{\mathfrak{X}},\breve{\mathfrak{b}}, \breve{\mathfrak{B}})$.

\subsection{Complexes for the relative homologies}
For $v,w\ge 0$, let $X_{vw} := M\ot B_w^n\ot$. By convenience we put $X_{vw} := 0$, otherwise. Consider the triple diagram
$$
\begin{tikzpicture}[xscale=0.9, yscale=0.75]
\begin{scope}[xshift=0cm,yshift=0cm]
\node at (-6,-4){$\mathcal{X}:=$};
\end{scope}
\begin{scope}[xshift=2cm,yshift=0cm]
\filldraw [black]  (-3,0)  circle (0.3pt); \filldraw [black]  (-3,-0.2)  circle (0.3pt); \filldraw [black]  (-3,-0.4)  circle (0.3pt);
\filldraw [black]  (0,0)  circle (0.3pt); \filldraw [black]  (0,-0.2)  circle (0.3pt); \filldraw [black]  (0,-0.4)  circle (0.3pt);
\filldraw [black]  (3,0)  circle (0.3pt); \filldraw [black]  (3,-0.2)  circle (0.3pt); \filldraw [black]  (3,-0.4)  circle (0.3pt);
\draw[->] (-3,-0.6) --  node[left=-2 pt] {$\scriptstyle -b$} (-3,-2.7); \draw[->] (0,-0.6) -- node[left=-2 pt] {$\scriptstyle b$} (0,-2.7); \draw[->] (3,-0.6) -- node[left=-2pt] {$\scriptstyle -b$} (3,-2.7);
\draw (-3,-3) node {$\scriptstyle X_{11}$}; \draw (0,-3) node {$\scriptstyle X_{11}$}; \draw (3,-3) node {$\scriptstyle X_{11}$};
\filldraw [black]  (-1.8,-0.8)  circle (0.3pt); \filldraw [black]  (-1.9,-1)  circle (0.3pt); \filldraw [black]  (-2,-1.2)  circle (0.3pt);
\draw[->] (-2.1,-1.4) --  node[left=-2 pt, pos=0.2] {$\scriptstyle -d'$} (-2.7,-2.7); \draw[->] (0.9,-1.4) --  node[left=-1 pt, pos=0.2] {$\scriptstyle d$} (0.3,-2.7); \draw[->] (3.9,-1.4) --  node[left=-2 pt, pos=0.2] {$\scriptstyle -d'$} (3.3,-2.7);
\filldraw [black]  (1.2,-0.8)  circle (0.3pt); \filldraw [black]  (1.1,-1)  circle (0.3pt); \filldraw [black]  (1,-1.2)  circle (0.3pt);
\filldraw [black]  (4.2,-0.8)  circle (0.3pt); \filldraw [black]  (4.1,-1)  circle (0.3pt); \filldraw [black]  (4,-1.2)  circle (0.3pt);
\filldraw [black]  (-4,-2)  circle (0.3pt); \filldraw [black]  (-4,-2.2)  circle (0.3pt); \filldraw [black]  (-4,-2.4)  circle (0.3pt);
\filldraw [black]  (-1,-2)  circle (0.3pt); \filldraw [black]  (-1,-2.2)  circle (0.3pt); \filldraw [black]  (-1,-2.4)  circle (0.3pt);
\filldraw [black]  (2,-2)  circle (0.3pt); \filldraw [black]  (2,-2.2)  circle (0.3pt); \filldraw [black]  (2,-2.4)  circle (0.3pt);
\draw (-4,-5) node {$\scriptstyle X_{10}$}; \draw (-1,-5) node {$\scriptstyle X_{10}$}; \draw (2,-5) node {$\scriptstyle X_{10}$};
\draw[->] (-4,-2.6) --  node[left=-2 pt] {$\scriptstyle -b$} (-4,-4.7); \draw[->] (-1,-2.6) --  node[left=-2 pt] {$\scriptstyle b$} (-1,-4.7); \draw[->] (2,-2.6) --  node[left=-2 pt] {$\scriptstyle -b$} (2,-4.7);
\draw[->] (-3.1,-3.4) --  node[left=-2 pt, pos=0.2] {$\scriptstyle -d'$} (-3.7,-4.7); \draw[->] (-0.1,-3.4) --  node[left=-1 pt, pos=0.2] {$\scriptstyle d$} (-0.7,-4.7);\draw[->] (2.9,-3.4) --  node[left=-2 pt, pos=0.2] {$\scriptstyle -d'$} (2.3,-4.7);
\draw (-3.4,-3) -- (-3.9,-3); \draw[->] (-4.1,-3)--node[above=-2 pt, pos=0.5] {$\scriptstyle \ide -t$} (-5.6,-3); \draw (-0.4,-3) -- (-0.9,-3);  \draw[->] (-1.1,-3)--node[above=-2 pt, pos=0.5] {$\scriptstyle N$} (-2.6,-3); \draw (2.6,-3) -- (2.1,-3); \draw[->] (1.9,-3)--node[above=-2 pt, pos=0.5] {$\scriptstyle \ide -t$} (0.4,-3); \draw[->] (5.6,-3)--node[above=-2 pt, pos=0.5] {$\scriptstyle N$} (3.4,-3);
\filldraw [black]  (-5.8,-3)  circle (0.3pt); \filldraw [black]  (-6,-3)  circle (0.3pt); \filldraw [black]  (-6.2,-3)  circle (0.3pt);
\filldraw [black]  (5.8,-3)  circle (0.3pt); \filldraw [black]  (6,-3)  circle (0.3pt); \filldraw [black]  (6.2,-3)  circle (0.3pt);
\draw[->] (-4.4,-5)--node[above=-2 pt, pos=0.2] {$\scriptstyle \ide -t$} (-6.6,-5); \draw[->] (-1.4,-5)--node[above=-2 pt, pos=0.2] {$\scriptstyle N$} (-3.6,-5); \draw[->] (1.6,-5)--node[above=-2 pt, pos=0.2] {$\scriptstyle \ide -t$} (-0.6,-5); \draw[->] (4.6,-5)--node[above=-2 pt, pos=0.2] {$\scriptstyle N$} (2.4,-5);
\filldraw [black]  (-6.8,-5)  circle (0.3pt); \filldraw [black]  (-7,-5)  circle (0.3pt); \filldraw [black]  (-7.2,-5)  circle (0.3pt);
\filldraw [black]  (4.8,-5)  circle (0.3pt); \filldraw [black]  (5,-5)  circle (0.3pt); \filldraw [black]  (5.2,-5)  circle (0.3pt);
\draw (-3,-6) node {$\scriptstyle X_{01}$}; \draw (0,-6) node {$\scriptstyle X_{01}$}; \draw (3,-6) node {$\scriptstyle X_{01}$};
\draw (-3,-3.4) --  node[left=-2 pt, pos=0.7] {$\scriptstyle -b$} (-3,-4.9); \draw[->] (-3,-5.1)  -- (-3,-5.7); \draw (0,-3.4) --  node[left=-2 pt, pos=0.7] {$\scriptstyle b$} (0,-4.9); \draw[->] (0,-5.1)  -- (0,-5.7); \draw (3,-3.4) --  node[left=-2 pt, pos=0.7] {$\scriptstyle -b$} (3,-4.9); \draw[->] (3,-5.1)  -- (3,-5.7);
\draw (-3.4,-6) -- (-3.9,-6); \draw[->] (-4.1,-6)--node[above=-2 pt, pos=0.5] {$\scriptstyle \ide -t$} (-5.6,-6); \draw (-0.4,-6) -- (-0.9,-6);  \draw[->] (-1.1,-6)--node[above=-2 pt, pos=0.5] {$\scriptstyle N$} (-2.6,-6); \draw (2.6,-6) -- (2.1,-6); \draw[->] (1.9,-6)--node[above=-2 pt, pos=0.5] {$\scriptstyle \ide -t$} (0.4,-6); \draw[->] (5.6,-6)--node[above=-2 pt, pos=0.5] {$\scriptstyle N$} (3.4,-6);
\filldraw [black]  (-5.8,-6)  circle (0.3pt); \filldraw [black]  (-6,-6)  circle (0.3pt); \filldraw [black]  (-6.2,-6)  circle (0.3pt);
\filldraw [black]  (5.8,-6)  circle (0.3pt); \filldraw [black]  (6,-6)  circle (0.3pt); \filldraw [black]  (6.2,-6)  circle (0.3pt);
\draw (-4,-8) node {$\scriptstyle X_{00}$}; \draw (-1,-8) node {$\scriptstyle X_{00}$}; \draw (2,-8) node {$\scriptstyle X_{00}$};
\draw[->] (-4.4,-8)--node[above=-2 pt, pos=0.5] {$\scriptstyle \ide -t$} (-6.6,-8); \draw[->] (-1.4,-8)--node[above=-2 pt, pos=0.5] {$\scriptstyle N$} (-3.6,-8); \draw[->] (1.6,-8)--node[above=-2 pt, pos=0.5] {$\scriptstyle \ide -t$} (-0.6,-8); \draw[->] (4.6,-8)--node[above=-2 pt, pos=0.5] {$\scriptstyle N$} (2.4,-8);
\draw[->] (-4,-5.4) --  node[left=-2 pt] {$\scriptstyle -b$} (-4,-7.7); \draw[->] (-1,-5.4) --  node[left=-2 pt] {$\scriptstyle b$} (-1,-7.7); \draw[->] (2,-5.4) --  node[left=-2 pt] {$\scriptstyle -b$} (2,-7.7);
\filldraw [black]  (-6.8,-8)  circle (0.3pt); \filldraw [black]  (-7,-8)  circle (0.3pt); \filldraw [black]  (-7.2,-8)  circle (0.3pt);
\filldraw [black]  (4.8,-8)  circle (0.3pt); \filldraw [black]  (5,-8)  circle (0.3pt); \filldraw [black]  (5.2,-8)  circle (0.3pt);
\filldraw [black]  (-1.8,-3.8)  circle (0.3pt); \filldraw [black]  (-1.9,-4)  circle (0.3pt); \filldraw [black]  (-2,-4.2)  circle (0.3pt);
\filldraw [black]  (1.2,-3.8)  circle (0.3pt); \filldraw [black]  (1.1,-4)  circle (0.3pt); \filldraw [black]  (1,-4.2)  circle (0.3pt);
\filldraw [black]  (4.2,-3.8)  circle (0.3pt); \filldraw [black]  (4.1,-4)  circle (0.3pt); \filldraw [black]  (4,-4.2)  circle (0.3pt);
\draw (-2.1,-4.4) --  node[left=-2 pt, pos=0.4] {$\scriptstyle -d'$} (-2.34,-4.92); \draw[->] (-2.43,-5.09) -- (-2.7,-5.7);
\draw (0.9,-4.4) --  node[left=-2 pt, pos=0.4] {$\scriptstyle d$} (0.66,-4.92); \draw[->] (0.57,-5.09) -- (0.3,-5.7);
\draw (3.9,-4.4) --  node[left=-2 pt, pos=0.4] {$\scriptstyle -d'$} (3.66,-4.92); \draw[->] (3.57,-5.09) -- (3.3,-5.7);
\draw[->] (-3.1,-6.4) --  node[left=-2 pt, pos=0.2] {$\scriptstyle -d'$} (-3.7,-7.7); \draw[->] (-0.1,-6.4) --  node[left=-1 pt, pos=0.2] {$\scriptstyle d$} (-0.7,-7.7);\draw[->] (2.9,-6.4) --  node[left=-2 pt, pos=0.2] {$\scriptstyle -d'$} (2.3,-7.7);
\end{scope}
\end{tikzpicture}
$$
where
\begin{alignat*}{2}
& b(\bx_0^n) := \sum_{j=0}^n \mu_j(\bx_0^n),&&\qquad\quad d'(\bx_0^n) := \sum_{j=0}^{n-1}\varrho_j(\bx_0^n),\\
&d(\bx_0^n) := d'(\bx_0^n) + \varrho_n(\bx_0^n), &&\qquad\quad N(\bx_0^n) := \sum_{l=0}^w t^l(\bx_0^n),
\end{alignat*}
the middle face $(X,b,d)$ is the $0$-th face and the bottom row is the $0$-th row. Note that the map $N\colon X_{v0}\longrightarrow X_{v0}$ is the identity map and the map $\ide - t\colon X_{v0}\longrightarrow X_{v0}$ is the zero map.

\smallskip

For $l\in \mathds{Z}$, let $\tau^l(\mathcal{X})$ be the subdiagram of $\mathcal{X}$ obtaining by deleting the faces $(X,b,d)$ and $(X,-b,d')$ placed in columns with index greater than $l$, and let $\tau_0(\mathcal{X})$ and $\tau_0^1(\mathcal{X})$ be the quotient triple diagrams of $\mathcal{X}$ by $\tau^{-1}(\mathcal{X})$ and $\tau^1(\mathcal{X})$ by $\tau^{-1}(\mathcal{X})$, respectively.

\begin{theorem}\label{th3.1} $\mathcal{X}$ is a triple complex. Moreover
\begin{align*}
&\HH^K_*(E,M) = \Ho_*(\Tot(\tau_0^1(\mathcal{X}))),\\
&\HC^K_*(E,M) = \Ho_*(\Tot(\tau_0(\mathcal{X}))),\\
&\HP^K_*(E,M) = \Ho_*(\Tot(\mathcal{X})),\\
&\HN^K_*(E,M) = \Ho_*(\Tot(\tau^1(\mathcal{X}))).
\end{align*}
\end{theorem}

\noindent Consequently, the following equalities hold:
\begin{alignat*}{3}
& d\xcirc b = - b\xcirc d,&&\qquad d'\xcirc b = - b\xcirc d',&&\qquad d'\xcirc N =  N\xcirc d,\\
& d\xcirc (\ide-t) = (\ide-t)\xcirc d',&&\qquad b\xcirc N = N\xcirc b,&&\qquad t\xcirc b = b\xcirc t.
\end{alignat*}
We will use these equalities freely throughout the paper.

\smallskip

Theorem~\ref{th3.1} is a consequence of Theorem~\ref{th3.2}, which we enounce below and whose proof will be relegated to Appendix~A. For $v,w\ge 0$, let $\hat{X}_{vw} := X_{vw}\oplus X_{v-1,w}$. Consider the diagram $(\hat{X},\hat{b},\hat{d})$ where the maps
$$
\hat{b}\colon \hat{X}_{vw} \longrightarrow \hat{X}_{v-1,w}\qquad\text{and}\qquad \hat{d}\colon \hat{X}_{vw} \longrightarrow \hat{X}_{v,w-1}
$$
are defined by
\[
\hat{b}(\bx,\byy) := \bigl(b(\bx) + (\ide-t)(\byy),-b(\byy)\bigr)\qquad\text{and}\qquad \hat{d}(\bx,\byy) := \bigl(d(\bx),-d'(\byy)\bigr).
\]
Note that $(\hat{X},\hat{b},\hat{d})$ is one of the partial total complexes of the triple complex $\tau_0^1(\mathcal{X})$, and so $\Tot(\tau_0^1(\mathcal{X})) = \Tot(\hat{X},\hat{b},\hat{d})$ (but we have not proven that $\tau_0^1(\mathcal{X})$ is a triple complex, yet). Let
$$
\hat{B}\colon \hat{X}_{vw}\longrightarrow \hat{X}_{v+1,w},\qquad \hat{\theta}\colon \hat{\mathfrak{X}}_{vw}\longrightarrow \hat{X}_{vw} \qquad\text{and}\qquad \hat{\vartheta}\colon \hat{X}_{vw}\longrightarrow \hat{\mathfrak{X}}_{vw}
$$
be the maps defined by
\begin{align*}
&\hat{B}(\bx,\byy) := (0,N(\bx)),\\
&\hat{\theta}(\bx,\byy) := (\bx + t(\byy),\mu_0^M(\byy)),\\
&\hat{\vartheta}(\bx,\byy) := \sum_{l=0}^{n-i(\byy)} \bigl(\bx,1\ot\mathfrak{t}^l(\byy)\bigr),
\end{align*}
where $\mathfrak{t}(y_0\ot\cdots\ot y_{n-1}) := (-1)^{n-1} y_{n-1}\ot y_0\ot\cdots\ot y_{n-2}$.

\begin{theorem}\label{th3.2} The following assertions hold:

\begin{enumerate}

\smallskip

\item $(\hat{X},\hat{b},\hat{d},\hat{B})$ is a double mixed complex.

\smallskip

\item Let $(\breve{X},\breve{b})$ be the total complex of $(\hat{X},\hat{b},\hat{d})$. The maps
\[\qquad\quad
\hat{\vartheta}\colon (\hat{X},\hat{b},\hat{d},\hat{B}) \longrightarrow (\hat{\mathfrak{X}},\hat{\mathfrak{b}}, \hat{\mathfrak{d}},\hat{\mathfrak{B}}) \qquad \text{and}\qquad \hat{\theta}\colon (\hat{\mathfrak{X}},\hat{\mathfrak{b}},\hat{\mathfrak{d}},\hat{\mathfrak{B}}) \longrightarrow (\hat{X},\hat{b},\hat{d},\hat{B}),
\]
are morphisms of double mixed complexes such that $\hat{\theta}\xcirc \hat{\vartheta} = \ide$. Moreover $\breve{\vartheta}\xcirc \breve{\theta}$ is homotopic to the identity map, where
\[\qquad\quad
\breve{\vartheta}\colon (\breve{X},\breve{b}) \longrightarrow (\breve{\mathfrak{X}},\breve{\mathfrak{b}}) \qquad \text{and}\qquad \breve{\theta}\colon (\breve{\mathfrak{X}},\breve{\mathfrak{b}}) \longrightarrow (\breve{X},\breve{b}),
\]
are the morphisms induced by $\hat{\vartheta}$ and $\hat{\theta}$, respectively.

\smallskip

\item The Hochschild, cyclic, periodic and negative homologies of $(\hat{X},\hat{b},\hat{d},\hat{B})$ are the Hoch\-schild, cyclic, periodic and negative homologies of the $K$-algebra $E$ relative to $M$, respectively.

\end{enumerate}
\end{theorem}

\begin{proof} See Appendix~A.
\end{proof}

\noindent {\bf Proof of Theorem~\ref{th3.1}.}\enspace Let $(\breve{X},\breve{b},\breve{B})$ be the mixed complex associated with $(\hat{X},\hat{b},\hat{d},\hat{B})$. Theorem~\ref{th3.1} follows immediately from Theorem~\ref{th3.2} and the fact that
\begin{align*}
&\Tot(\tau_0^1(\mathcal{X})) = (\breve{X},\breve{b}),\\
&\Tot(\tau_0(\mathcal{X})) = \Tot(\BC(\breve{X},\breve{b},\breve{B})),\\
&\Tot(\mathcal{X}) = \Tot(\BP(\breve{X},\breve{b},\breve{B})),\\
&\Tot(\tau^1(\mathcal{X})) = \Tot(\BN(\breve{X},\breve{b},\breve{B})),
\end{align*}
which can be easily checked.\qed

\medskip

The following result was proven in~\cite{G-G}.

\begin{lemma}\label{le3.3} The rows of the triple complex $\mathcal{X}$ are contractible.
\end{lemma}

\begin{proof} For $v,w\ge 0$, let $\sigma,\sigma'\colon X_{vw}\longrightarrow X_{vw}$ be the maps defined by
$$
\sigma := \frac{1} {w+1}\ide\qquad\text{and}\qquad \sigma' := \sum_{j=0}^{w-1} \frac{w-j}{w+1} t^j.
$$
A direct computation shows that:
\allowdisplaybreaks
\begin{align}
&\sigma\xcirc N = N\xcirc\sigma = \frac{1}{w+1} N,\label{*eq4}\\
& (\ide-t)\xcirc \sigma' = \sigma'\xcirc (\ide-t) = \sum_{j=0}^{w-1} \frac{w-j}{w+1} t^j - \sum_{j=1}^w \frac{w-j+1}{w+1} t^j = \ide - \frac{1}{w+1}N.\label{*eq5}
\end{align}
The result follows immediately from these equalities.
\end{proof}

\begin{theorem}\label{th3.4} Let $(\ov{X},\ov{b},\ov{d})$ be the cokernel of $\ide-t\colon (X,-b,-d') \longrightarrow (X,b,d)$. The relative cyclic homology $\HC^K_*(E,M)$ is the homology of $(\ov{X},\ov{b},\ov{d})$.
\end{theorem}

\begin{proof} This follows immediately from Theorem~\ref{th3.1} and Lemma~\ref{le3.3}.
\end{proof}

\subsection{Graded algebras} Let $E = A_0\oplus A_1\oplus A_2\oplus\cdots$ be a graded algebra. Let $M:= A_1\oplus A_2\oplus\cdots$ be the augmentation ideal of $E$. Clearly $E$ is a cleft extension of $A:=A_0$ by $M$. Let $K$ be a $k$-subalgebra of $A$. As it is well known, the double mixed complex $(\hat{\mathfrak{X}}, \hat{\mathfrak{b}},\hat{\mathfrak{d}},\hat{\mathfrak{B}})$ decompose as the direct sum
\begin{equation}
(\hat{\mathfrak{X}},\hat{\mathfrak{b}},\hat{\mathfrak{d}},\hat{\mathfrak{B}}) = \bigoplus_{i\in \mathds{N}} (\hat{\mathfrak{X}}^{(i)},\hat{\mathfrak{b}},\hat{\mathfrak{d}},\hat{\mathfrak{B}}),\label{*eq6}
\end{equation}
where $\hat{\mathfrak{X}}^{(i)}_{vw}$ is the $k$-submodule of $\hat{\mathfrak{X}}_{vw}$ generated by the simple tensors $a_0\ot\cdots\ot a_n$ such that each $a_j$ is homogeneous of degree $|a_j|$ and $\sum_j |a_j|  = i$. Similarly, the double mixed complex $(\hat{X},\hat{b},\hat{d},\hat{B})$ decompose as the direct sum
$$
(\hat{X},\hat{b},\hat{d},\hat{B}) = \bigoplus_{i\in \mathds{N}} (\hat{X}^{(i)},\hat{b},\hat{d},\hat{B}),
$$
where $\hat{X}^{(i)}$ is defined in the same manner as $\hat{\mathfrak{X}}^{(i)}_{vw}$,  and the morphisms $\hat{\vartheta}$ and $\hat{\theta}$ are compatible with these decompositions. Moreover these maps induce homotopy equivalence between the complexes $(\hat{X}^{(i)},\hat{b},\hat{d},\hat{B})$ and $(\hat{\mathfrak{X}}^{(i)},\hat{\mathfrak{b}},\hat{\mathfrak{d}},\hat{\mathfrak{B}})$. In order to check this fact it suffices to note that the homotopy $\breve{\epsilon}$, introduced in item~4) of Lemma~\ref{leA.3}, is compatible with the decomposition~\eqref{*eq6}.

\section{The harmonic decomposition}\label{The harmonic decomposition}
As in the proof of Theorem~\ref{th3.1}, we let $(\breve{X},\breve{b}, \breve{B})$ denote the mixed complex associated with the double mixed complex $(\hat{X},\hat{b},\hat{d}, \hat{B})$, introduced in Theorem~\ref{th3.2}. Our purpose in this section is to show that $(\breve{X},\breve{b}, \breve{B})$ has a harmonic decomposition like the one studied in \cite{C-Q2}. In order to carry out this task we need to define a de Rham coboundary map and a Karoubi operator on $(\breve{X},\breve{b})$. As we said in the introduction we are going to work with a new double mixed complex $(\ddot{X},\ddot{b},\ddot{d},\ddot{B})$, whose associated mixed complex is also $(\breve{X},\breve{b},\breve{B})$. In the first three subsections we follow closely the exposition of \cite{C-Q2}.

\subsection{The Rham coboundary map and the Karoubi operator} It is easy to see that $\tau_0^1(\mathcal{X})$ is the total complex of the double complex
$$
(\ddot{X}{,\hspace{1pt}} \ddot{b}{,\hspace{1pt}}\ddot{d}):=
\begin{tikzcd}[column sep=0.8cm,row sep=0.8cm]
&& \vdots  \arrow{d}{\ddot{b}} & \vdots \arrow{d}{\ddot{b}}& \vdots \arrow{d}{\ddot{b}}\\
&\dots \arrow{r}{\ddot{d}} & \ddot{X}_{22} \arrow{r}{\ddot{d}} \arrow{d}{\ddot{b}} & \ddot{X}_{21} \arrow{r}{\ddot{d}} \arrow{d}{\ddot{b}} & \ddot{X}_{20} \arrow{d}{\ddot{b}}\\
&\dots \arrow{r}{\ddot{d}} & \ddot{X}_{12} \arrow{r}{\ddot{d}} \arrow{d}{\ddot{b}} & \ddot{X}_{11} \arrow{r}{\ddot{d}} \arrow{d}{\ddot{b}} & \ddot{X}_{10} \arrow{d}{\ddot{b}}\\
&\dots \arrow{r}{\ddot{d}} & \ddot{X}_{02} \arrow{r}{\ddot{d}} & \ddot{X}_{01} \arrow{r}{\ddot{d}} & \ddot{X}_{00},
\end{tikzcd}
$$
where $\ddot{X}_{vw} := X_{vw}\oplus X_{v,w-1}$ and the boundary maps are defined by
\[
\ddot{b}(\bx,\byy) := (b(\bx),-b(\byy))\qquad\text{and}\qquad \ddot{d}(\bx,\byy) := (d(\bx)+ (\ide-t)(\byy),-d'(\byy)).
\]
The de Rham coboundary map  $\ddot{d}\sR\colon \ddot{X}_{vw}\longrightarrow \ddot{X}_{v,w+1}$ is defined by $\ddot{d}\sR(\bx,\byy):=(0,\bx)$. It is obvious that $(\ddot{X},\ddot{d}\sR)$ is acyclic. We now define the Karoubi operator of $\ddot{X}$. Let
$$
\ddot{\kappa}(0)\colon \ddot{X}_{vw}\longrightarrow \ddot{X}_{vw}\qquad\text{and}\qquad \ddot{\kappa}(1)\colon \ddot{X}_{vw} \longrightarrow \ddot{X}_{vw}
$$
be the maps defined by
\[
\ddot{\kappa}(0)(\bx,\byy) := (t(\bx),t(\byy))\qquad\text{and}\qquad \ddot{\kappa}(1)(\bx,\byy) := \bigl(0,d'(\bx) -d(\bx)\bigr).
\]
The Karoubi operator $\ddot{\kappa}$ of $\ddot{X}$ is the degree zero operator defined by
\[
\ddot{\kappa} := \ddot{\kappa}(0) + \ddot{\kappa}(1).
\]
Let $\breve{d}\sR\colon \breve{X}_n\longrightarrow \breve{X}_{n+1}$ and $\breve{\kappa}\colon \breve{X}_n\to \breve{X}_n$ be the maps defined by
$$
\breve{d}\sdR_n := \bigoplus_{w=0}^n \ddot{d}\sdR_{n-w,w}\qquad\text{and}\qquad\breve{\kappa}_n := \bigoplus_{w=0}^n \ddot{\kappa}_{n-w,w},
$$
respectively. A direct computation shows that
\begin{equation}
\ide - \ddot{\kappa} = \ddot{d} \xcirc \ddot{d}\sdR+ \ddot{d}\sdR \xcirc \ddot{d} \qquad\text{and}\qquad 0 = \ddot{b} \xcirc \ddot{d}\sdR+ \ddot{d}\sdR \xcirc \ddot{b}.\label{*eq7}
\end{equation}
In particular, $\ddot{\kappa}$ is homotopic to the identity with respect to either of the differentials $\ddot{d}$, $\ddot{d}\sR$, and so it commutes with them. From~\eqref{*eq7} it follows that
\[
\ide - \breve{\kappa} = \breve{b} \xcirc  \breve{d}\sdR +  \breve{d}\sdR \xcirc \breve{b}.
\]
Consequently $\breve{\kappa}$ commutes with $\breve{b}$ and $\breve{d}\sR$. Hence $\ddot{\kappa}$ also commutes with $\ddot{b}$ (which can be also proven by a direct computation). Let $\ddot{B}\colon \ddot{X}_{vw}\longrightarrow \ddot{X}_{v,w+1}$ be the map defined by $\ddot{B}(\bx,\byy) := (0,N(\bx))$. An easy computation shows that $(\ddot{X},\ddot{b},\ddot{d},\ddot{B})$ is a double mixed complex and that its associated mixed complex is $(\breve{X},\breve{b},\breve{B})$. Furthermore,
$$
\ddot{B}(\bx) = \sum_{i=0}^w \ddot{\kappa}^i \xcirc \ddot{d}\sR(\bx) \qquad\text{for all $\bx\in \ddot{X}_{vw}$.}
$$
Using this we obtain
\[
\ddot{B}\xcirc  \ddot{\kappa} =  \ddot{\kappa}\xcirc \ddot{B} = \ddot{B} \qquad\text{and}\qquad \ddot{d}\sdR\xcirc \ddot{B} = \ddot{B} \xcirc  \ddot{d}\sdR = 0.
\]

\subsection{The harmonic decomposition} From the definition of $\ddot{\kappa}$ it follows immediately that
\[
(\ddot{\kappa}^w-\ide) \xcirc (\ddot{\kappa}^{w+1}-\ide)(\ddot{X}_{vw}) \subseteq (\kappa^w-\ide)(0\oplus X_{v,w-1}) = 0.
\]
This implies that $\ddot{\kappa}$ satisfies the polynomial equation $P_w(\ddot{\kappa}) = 0$ on $\ddot{X}_{vw}$, where
\[
P_w = (X^{w+1}-1)(X^w-1).
\]
The roots of $P_w$ are the $r$-th roots of unity, with $r = w+1$ and $r=w$. Moreover, $1$ is a double root and the all other roots are simple. Consequently $\ddot{X}_{vw}$ decomposes into the direct sum of the generalized eigenspace $\ker(\ddot{\kappa} - \ide)^2$ and its complement $\ima(\ddot{\kappa} - \ide)^2$. Combining this for all $v,w$ we obtain the following decomposition
\[
\ddot{X} = \ker(\ddot{\kappa} - \ide)^2 \oplus\ima(\ddot{\kappa} - \ide)^2,
\]
Each of these generalized subspaces is stable under any operator commuting with $\ddot{\kappa}$, for instance, $\ddot{b}$, $\ddot{d}$, $\ddot{d}\sR$ and $\ddot{B}$.

\subsection{The harmonic projection and the Green operator}
Let $P$ be the harmonic projection operator, which is the identity map on $\ker(\ddot{\kappa} - \ide)^2$ and the zero map on $\ima(\ddot{\kappa} - \ide)^2$. Thus we have
\[
\ddot{X} = P(\ddot{X})\oplus P^{\perp}(\ddot{X}),
\]
where $P^{\perp} := \ide-P$. It is clear that $(P(\ddot{X}),\ddot{b},\ddot{d},\ddot{B})$ and $(P^{\perp}(\ddot{X}),\ddot{b},\ddot{d},\ddot{B})$ are double mixed subcomplexes of $(\ddot{X},\ddot{b},\ddot{d},\ddot{B})$. On $P^{\perp}(\ddot{X})$ the operator
$$
\ide - \ddot{\kappa} = \ddot{d} \xcirc \ddot{d}\sR + \ddot{d}\sR \xcirc \ddot{d}
$$
is both invertible and homotopic to zero with respect to either differential $\ddot{d}$ and $\ddot{d}\sR$. Hence the complexes $(P^{\perp}(\ddot{X}),\ddot{d})$ and$(P^{\perp}(\ddot{X}),\ddot{d}\sR)$ are acyclic. Let
\[
P(\breve{X}_n) = \bigoplus_{w=0}^n P(\ddot{X}_{n-w,w})\qquad\text{and}\qquad P^{\perp}(\breve{X}_n) = \bigoplus_{w=0}^n P^{\perp}(\ddot{X}_{n-w,w}).
\]
The same argument shows that $(P^{\perp}(\breve{X}),\breve{b})$ and $(P^{\perp}(\breve{X}),\breve{d}\sdR)$ are also acyclic. The Green operator $G\colon \ddot{X}\to \ddot{X}$ is defined to be zero on $P(\ddot{X})$ and the inverse of $\ide-\ddot{\kappa}$ on $P^{\perp}(\ddot{X})$. It is clear that
\begin{equation}
G \xcirc P = P\xcirc G = 0\qquad\text{and}\qquad P^{\perp} = G\xcirc (\ide-\ddot{\kappa}) = G\xcirc (\ddot{d} \xcirc \ddot{d}\sdR + \ddot{d}\sdR \xcirc \ddot{d}).\label{*eq8}
\end{equation}
Moreover $P$ and $G$ commute with each operator that commutes with $\ddot{\kappa}$.

\begin{proposition}\label{pr4.1} The following equality holds:
\[
P^{\perp}(\ddot{X}) = \ddot{d}\xcirc P^{\perp}(\ddot{X}) \oplus \ddot{d}\sR\xcirc P^{\perp}(\ddot{X}).
\]
Furthermore $\ddot{d}\sR$ maps $\ddot{d}\xcirc P^{\perp}(\ddot{X})$ isomorphically onto $\ddot{d}\sR\xcirc P^{\perp}(\ddot{X})$ with inverse $G\xcirc \ddot{d}$ and $\ddot{d}$ maps $\ddot{d}\sR\xcirc P^{\perp}(\ddot{X})$ isomorphically onto $\ddot{d}\xcirc P^{\perp}(\ddot{X})$ with inverse $G\xcirc \ddot{d}\sR$.
\end{proposition}

\begin{proof} The proof of \cite{C-Q2}*{Proposition~2.1} works in our setting.
\end{proof}

The above proposition gives a new proof that $(P^{\perp}(\ddot{X}),\ddot{d})$ and $(P^{\perp}(\ddot{X}),\ddot{d}\sR)$ are acyclic.

\begin{proposition}\label{pr4.2} An element $\bx\in \ddot{X}_{vw}$ belongs to $P(\ddot{X}_{vw})$ if and only if $\ddot{d}\sR(\bx)$ and $\ddot{d}\sR \xcirc \ddot{d}(\bx)$ are $\ddot{\kappa}$-invariant.
\end{proposition}

\begin{proof} The proof of \cite{C-Q2}*{Proposition~2.2} works in our setting.
\end{proof}

For $v,w\ge 0$ let $\grave{X}_{vw}$ and $\acute{X}_{vw}$ be the image of the canonical inclusions of $X_{vw}$ into $\ddot{X}_{vw}$ and $\ddot{X}_{v,w+1}$ respectively, and let $\acute{\kappa}\colon \acute{X}_{vw}\longrightarrow \acute{X}_{vw}$ be the map induced by $\ddot{\kappa}$. For $\bx = (\bx_0,0)\in \grave{X}_{vw}$ we write
\[
\grave{b}(\bx) := (b(\bx_0),0) \qquad\text{and}\qquad \grave{d}(\bx) := (d(\bx_0),0),
\]
and, for $\byy = (0,\byy_0)\in \acute{X}_{vw}$, we write
\[
\acute{b}(\byy) := (0,b(\byy_0)),\qquad \acute{d}'(\byy) := (0,d'(\byy_0))\qquad\text{and}\qquad \acute{t}(\byy) := (0,t(\byy_0)).
\]
It is obvious that $\acute{\kappa}$ coincides with $\acute{t}$. Note that $\ddot{\kappa}$ has finite order on $\ddot{d}\sR(\ddot{X}) = \acute{X}$ in each degree. In fact $\ddot{\kappa}^{w+1} = \ide$ on $\acute{X}_{vw}$. By the discussion in the page~86 of \cite{C-Q2},
\begin{align}
& \acute{X}_{vw} = \ker(\ide-\acute{\kappa})\oplus \ima(\ide-\acute{\kappa}),\label{*eq9}\\
& P(\ddot{d}\sdR(\grave{X}_{vw})) = P(\acute{X}_{vw})=\ker(\ide-\acute{\kappa}),\\
& P^{\perp}(\ddot{d}\sdR(\grave{X}_{vw}))=P^{\perp}(\acute{X}_{vw})= \ima(\ide-\acute{\kappa}),
\end{align}
and the maps $P_{\acute{X}_{vw}}$ and $G_{\acute{X}_{vw}}$, defined as the projection onto $\ker(\ide-\acute{\kappa})$ associated with the decomposition~\eqref{*eq9} and the Green operator for $\ide-\acute{\kappa}\colon \acute{X}_{vw}\longrightarrow \acute{X}_{vw}$, respectively, satisfy:
\begin{equation}
P_{\acute{X}_{vw}} = \frac{1}{w+1} \sum_{i=0}^w \acute{\kappa}^i\qquad\text{and}\qquad G_{\acute{X}_{vw}} = \frac{1}{w+1} \sum_{i=0}^w \left(\frac{w}{2}-i\right)\acute{\kappa}^i.\label{eq10}
\end{equation}
Consequently, for all $\bx\in \ddot{X}_{vw}$,
\begin{align}
& P\xcirc \ddot{d}\sdR(\bx) = \frac{1}{w+1}\sum_{i=0}^w \ddot{\kappa}^i\xcirc \ddot{d}\sdR(\bx) = \frac{1}{w+1} \ddot{B}(\bx)\label{eq11}\\
\shortintertext{and}
& G\xcirc \ddot{d}\sdR(\bx) = \frac{1}{w+1}\sum_{i=0}^w \left(\frac{w}{2}-i\right) \ddot{\kappa}^i\xcirc \ddot{d}\sdR(\bx).\label{eq12}
\end{align}
Formula~\eqref{eq11} implies that
\begin{equation}
\ddot{B}(P^{\perp}(\ddot{X}))=0\label{eq13}
\end{equation}
and
$$
\ddot{B}(\bx) = (w+1) \ddot{d}\sdR(P(\bx))\qquad\text{for all $\bx\in\ddot{X}_{vw}$.}
$$
So, since $(P^{\perp}(\ddot{X}),\ddot{d}\sR)$ is acyclic,
\begin{equation}
H_*(P(\ddot{X}),\ddot{B}) = H_*(P(\ddot{X}),\ddot{d}\sdR) = H_*(\ddot{X},\ddot{d}\sdR) = 0.\label{eq14}
\end{equation}
In the terminology of \cite{C-Q2} this says that $(P(\ddot{X}),\ddot{d},\ddot{b}, \ddot{B})$ is $\ddot{B}$-acyclic. Lastly,~\eqref{eq10} combined with~\eqref{eq12} and the second formula of~\eqref{*eq8}, allows us to obtain an explicit formula for $P$. In fact, for $\bx\in \acute{X}_{vw}$, this is given by the first equality in~\eqref{eq10}. Hence, assume $\bx\in\grave{X}_{vw}$. Since by~\eqref{eq12} we know that $G \xcirc \ddot{d}\sdR(\bx)\in \acute{X}_{vw}$, we have:
\[
G\xcirc \ddot{d}\xcirc \ddot{d}\sdR(\bx) = \ddot{d}\xcirc G\xcirc \ddot{d}\sdR(\bx) = -\acute{d}' \xcirc G \xcirc \ddot{d}\sdR(\bx) + \sw\xcirc (\ide - \acute{t})\xcirc G \xcirc \ddot{d}\sdR(\bx),
\]
where $\sw\colon \acute{X}_{vw}\longrightarrow \grave{X}_{vw}$ is the map defined by $\sw(0,\bx) := (\bx,0)$. Using this, the second equalities in~\eqref{*eq8} and~\eqref{eq10}, and the fact that $\acute{t}\xcirc \ddot{\kappa}^i \xcirc \ddot{d}\sR(\bx) = \ddot{\kappa}^{i+1}\xcirc \ddot{d}\sR(\bx)$, we obtain:

\begin{enumerate}

\item If $w=0$, then $P(\bx) = \bx$.

\smallskip

\item If $w>0$, then

\begin{align*}
P(\bx) & =  \bx -  G\xcirc \ddot{d}\sdR\xcirc \grave{d}(\bx) - G\xcirc \ddot{d}\xcirc \ddot{d}\sdR(\bx)\\
& = \bx - \frac{1}{w}\sum_{i=0}^{w-1} \left(\frac{w-1}{2}-i\right) \ddot{\kappa}^i\xcirc \ddot{d}\sdR \xcirc \grave{d}(\bx)\\
& + \frac{1}{w+1}\sum_{i=0}^w \left(\frac{w}{2}-i\right) \acute{d}'\xcirc \ddot{\kappa}^i \xcirc \ddot{d}\sdR(\bx)\\
& - \frac{1}{w+1}\sum_{i=0}^w \left(\frac{w}{2}-i\right)\sw\xcirc(\ide-\acute{t})\xcirc \ddot{\kappa}^i \xcirc \ddot{d}\sdR(\bx)\displaybreak[0] \\
& = \frac{1}{w+1}\sw\xcirc\ddot{B}_n(\bx) - \frac{1}{w}\sum_{i=0}^{w-1} \left(\frac{w-1}{2}-i\right) \ddot{\kappa}^i\xcirc \ddot{d}\sdR \xcirc \grave{d}(\bx)\\
&+ \frac{1}{w+1}\sum_{i=0}^w \left(\frac{w}{2}-i\right) \acute{d}'\xcirc \ddot{\kappa}^i \xcirc \ddot{d}\sdR(\bx).
\end{align*}

\smallskip

\end{enumerate}
Summarizing,
\begin{align}
P(0,\byy) & = \frac{1}{w}(0,N(\byy)) && \qquad\text{for $(0,\byy)\in \ddot{X}_{vw}$ with $w>0$,}\label{e1}\\
P(\bx,0) &= (\bx,0) && \qquad\text{for $(\bx,0)\in \ddot{X}_{v0}$,}\label{e2}\\
P(\bx,0) & = \frac{1}{w+1}\bigl(N(\bx),0\bigr)\nonumber\\
& - \frac{1}{w}\sum_{i=0}^{w-1} \left(\frac{w-1}{2}-i\right) \bigl(0,t^i(d(\bx))\bigr) && \qquad\text{for $(\bx,0)\in \ddot{X}_{vw}$ with $w>0$.}\label{e3}\\
& + \frac{1}{w+1}\sum_{i=0}^w \left(\frac{w}{2}-i\right) \bigl(0,d'(t^i(\bx))\bigr).\nonumber
\end{align}

\begin{remark}\label{info sobre P(X)} Take $\bx := (\bx_0,\bx_1)\in \ddot{X}_{vw}$, with $\bx_0\in X_{vw}$ and $\bx_1\in X_{v,w-1}$. By Proposition~\ref{pr4.2} we know that $\bx\in P(\ddot{X})$ if and only if $\bx_0$ and $d(\bx_0)+(\ide-t)(\bx_1)$ are $t$-invariant. From this it follows immediately that if $\bx\in P(\ddot{X})$, then $(\bx_0,\bx'_1)\in P(\ddot{X})$ for all $\bx'_1\in X_{v,w-1}$ such that $\bx'_1-\bx_1$ is a $t$-invariant element. Conversely, if $\bx$ and $(\bx_0,\bx'_1)$ belong to $P(\ddot{X})$, then $\bz:=(\ide-t)(\bx'_1 - \bx_1)$ is $t$-invariant, but this implies that
$$
w\bz = N(\bz) = N\xcirc (\ide-t)(\bx'_1 - \bx_1) = 0.
$$
In other words, that $\bx'_1-\bx_1$ is $t$-invariant (note that if $w=0$, then $\bx'_1=\bx_1=0$). Let ${}^t\!P(\ddot{X}_{vw})$ be the set of all elements of the shape $(0,\bx_1)\in \ddot{X}_{vw}$ with $\bx_1$ a $t$-invariant element. By the previous computations
$$
{}^t\!P(\ddot{X}_{vw}) = P(\ddot{X}_{vw})\cap \acute{X}_{v,w-1}.
$$
It is evident that $({}^t\!P(\ddot{X}),-\acute{b},-\acute{d}')$ is a subcomplex of $(P(\ddot{X}),\ddot{b},\ddot{d})$.
\end{remark}

\begin{proposition}\label{pr4.3} Assume that $w>0$. Then
$$
P(t(\bx),0) = P(\bx,0)+ \frac{1}{w} (0,N\xcirc \varrho_n(\bx)) - \frac{1}{w(w+1)} (0,N\xcirc d(\bx))\quad\text{and} \quad P(0,t(\byy)) = P(0,\byy)
$$
for all $(\bx,0),(0,\byy)\in \ddot{X}_{vw}$.
\end{proposition}

\begin{proof} By equality~\eqref{e2}, for $(0,\byy)\in \ddot{X}_{vw}$ we have
$$
P(0,\byy) = \frac{1}{w}(0,N(\byy)) = \frac{1}{w}(0,N(t(\byy))) = P(0,t(\byy)).
$$
Next we obtain the formula for $P(t(\bx),0)$. By equality~\eqref{e3},
\begin{align*}
P(t(\bx),0)- P(\bx,0) & = \frac{1}{w}\sum_{i=0}^{w-1} \left(\frac{w-1}{2}-i\right) \bigl(0,(t^i-t^{i+1})\xcirc d(\bx)\bigr)\\
& + \frac{1}{w+1}\sum_{i=0}^w \left(\frac{w}{2}-i\right) \bigl(0,d'\xcirc (t^{i+1}-t^i)(\bx)\bigr)\\
& = \frac{1}{w}\sum_{i=0}^{w-1} \left(\frac{w-1}{2}-i\right) \bigl(0,(t^i-t^{i+1})\xcirc d(\bx)\bigr)\\
& + \frac{1}{w+1}\sum_{i=0}^w \left(\frac{w}{2}-i\right) \bigl(0,d\xcirc (t^{i+1}-t^i)(\bx)\bigr)\\
& - \frac{1}{w+1}\sum_{i=0}^w \left(\frac{w}{2}-i\right) \bigl(0,\varrho_n\xcirc (t^{i+1}-t^i)(\bx)\bigr)\\
& = \frac{1}{w}\sum_{i=0}^{w-1} \left(\frac{w-1}{2}-i\right) \bigl(0,(\ide-t)\xcirc t^i \xcirc d(\bx)\bigr)\\
& + \frac{1}{w+1}\sum_{i=0}^w \left(\frac{w}{2}-i\right) \bigl(0,(t-\ide)\xcirc d'\xcirc t^i(\bx)\bigr)\\
& + \bigl(0,\varrho_n(\bx)\bigr) - \frac{1}{w+1}\bigl(0,\varrho_n\xcirc N(\bx)\bigr),
\end{align*}
because
$$
\sum_{i=0}^w \left(\frac{w}{2}-i\right) (t^{i+1}-t^i) = -(w+1) t^0 + N.
$$
Hence
$$
P(t(\bx),0)- P(\bx,0) = \bigl(0,(t-\ide)(\bx')\bigr)+\bigl(0,\varrho_n(\bx)\bigr) - \frac{1}{w+1} \bigl(0,\varrho_n\xcirc N(\bx)\bigr),
$$
for some $\bx'\in X_{v,w-1}$. But by Remark~\ref{info sobre P(X)} we know that
$$
(t-\ide)(\bx') + \varrho_n(\bx) - \frac{1}{w+1} \varrho_n\xcirc N(\bx)
$$
is $t$-invariant, and so
\begin{align*}
(t-\ide)(\bx') + \varrho_n(\bx) - \frac{1}{w+1} \varrho_n\xcirc N(\bx) & = \frac{1}{w}N\xcirc \varrho_n(\bx) - \frac{1}{w(w+1)} N\xcirc\varrho_n\xcirc N(\bx)\\
& = \frac{1}{w}N\xcirc \varrho_n(\bx) - \frac{1}{w(w+1)} N\xcirc d(\bx).
\end{align*}
The desired formula for $P(t(\bx),0)$ follows immediately from this fact.
\end{proof}

\begin{proposition}\label{pr4.3'} For $1\le i\le w$ and $(\bx,0)\in \ddot{X}_{vw}$, we have
\begin{equation}\label{*eq10}
P(t^i(\bx),0) = P(\bx,0) + \frac{1}{w}\sum_{j=0}^{i-1} (0,N\xcirc \varrho_n\xcirc t^j(\bx)) - \frac{i}{w(w+1)} (0,N\xcirc d(\bx)).
\end{equation}
\end{proposition}

\begin{proof} By induction on $i$. Case $i=1$ is Proposition~\ref{pr4.3}. Assume $0<i<w$ and that formula~\eqref{*eq10} is valid for $i$. Then, again by Proposition~\ref{pr4.3},
\begin{align*}
P(t^{i+1}(\bx),0) & = P(t^i(\bx),0)+ \frac{1}{w} (0,N\xcirc \varrho_n\xcirc t^i(\bx)) - \frac{1}{w(w+1)} (0,N\xcirc d\xcirc t^i(\bx))\\
& = P(t^i(\bx),0)+ \frac{1}{w} (0,N\xcirc \varrho_n\xcirc t^i(\bx)) - \frac{1}{w(w+1)} (0,N\xcirc d(\bx)),
\end{align*}
where the last equality follows from the fact that $N\xcirc d\xcirc t^i = d'\xcirc N\xcirc t^i = d'\xcirc N = N\xcirc d$. Consequently, by the inductive hypothesis,
$$
P(t^{i+1}(\bx),0) = P(\bx,0) + \frac{1}{w}\sum_{j=0}^i (0,N\xcirc \varrho_n\xcirc t^j(\bx)) - \frac{i+1}{w(w+1)} (0,N\xcirc d(\bx)),
$$
as desired.
\end{proof}

\begin{corollary}\label{cor4.4} For $w\ge 1$ and $(\bx,0)\in \ddot{X}_{vw}$, we have
$$
P(\bx,0) = \frac{1}{w+1}(N(\bx),0) + \sum_{j=0}^{w-1}\sum_{i=0}^w \frac{2j+2i -2w + 1}{2w(w+1)} \bigl(0,t^j\xcirc \varrho_n\xcirc t^i(\bx)\bigr).
$$
\end{corollary}

\begin{proof} By Proposition~\ref{pr4.3'},
\begin{align*}
P(\bx,0) & = \frac{1}{w+1}  P\bigl(N(\bx),0\bigr) - \frac{1}{(w+1)w}\sum_{i=0}^w (w-i) \bigl(0,N\xcirc \varrho_n\xcirc t^i(\bx)\bigr) + \frac{1}{2(w+1)} \bigl(0,N\xcirc d(\bx)\bigr)\\
& = \frac{1}{w+1}  P(N(\bx),0) + \sum_{i=0}^w \frac{2i - w}{2w(w+1)} \bigl(0,N\xcirc \varrho_n\xcirc t^i(\bx)\bigr).
\end{align*}
On the other hand, by formula~\eqref{e3}
\begin{align*}
\frac{1}{w+1} P(N(\bx),0) & = \frac{1}{w+1}\bigl(N(\bx),0\bigr) - \frac{1}{(w+1)w}\sum_{j=0}^{w-1} \left(\frac{w-1}{2}-j\right) \bigl(0,t^j\xcirc d\xcirc N(\bx)\bigr)\\
& + \frac{1}{(w+1)^2}\sum_{j=0}^w \left(\frac{w}{2}-j\right) \bigl(0,d'\xcirc N(\bx)\bigr)\\
& = \frac{1}{w+1}\bigl(N(\bx),0\bigr) - \frac{1}{(w+1)w}\sum_{j=0}^{w-1} \left(\frac{w-1}{2}-j\right) \bigl(0,t^j\xcirc d'\xcirc N(\bx)\bigr)\\
& - \frac{1}{(w+1)w}\sum_{j=0}^{w-1} \left(\frac{w-1}{2}-j\right) \bigl(0,t^j\xcirc \varrho_n\xcirc N(\bx)\bigr)\\
& = \frac{1}{w+1}\bigl(N(\bx),0\bigr) - \frac{1}{(w+1)w}\sum_{j=0}^{w-1} \left(\frac{w-1}{2}-j\right) \bigl(0,t^j\xcirc \varrho_n\xcirc N(\bx)\bigr),
\end{align*}
where the second equality follows from the facts that $d = d'+\varrho_n$ and $\sum_{i=0}^w \left(\frac{w}{2}-i\right)=0$; and the last equality, from the facts that $t^i\xcirc d'\xcirc N = t^i\xcirc N\xcirc d = N\xcirc d$ and $\sum_{i=0}^{w-1} \left(\frac{w-1}{2}-i\right)=0$. So,
$$
P(\bx,0) = \frac{1}{w+1} (N(\bx),0) + \sum_{j=0}^{w-1}\sum_{i=0}^w \frac{2j+2i -2w + 1}{2w(w+1)} (0,t^j\xcirc \varrho_n\xcirc t^i(\bx)\bigr),
$$
as desired
\end{proof}

We now consider the chain complex $(\ddot{X},\ddot{b},\ddot{d})$ and we denote by $\ker(\ddot{B})$ and $\ima(\ddot{B})$ the kernel and image of $\ddot{B}$ respectively. These are subcomplexes of $(\ddot{X},\ddot{b},\ddot{d})$. By~\eqref{eq13} and~\eqref{eq14}, we have
$$
\ker(\ddot{B})/\ima(\ddot{B}) = P^{\perp}(\ddot{X}).
$$
Consequently,
$$
H_*(\ker(\ddot{B})/\ima(\ddot{B}),\ddot{b},\ddot{d}) = 0.
$$
That is, the double mixed complex $(\ddot{X},\ddot{d},\ddot{b},\ddot{B})$ has the Connes property (\cite{C-Q2}).

\smallskip

Let us define the reduced cyclic complex $\ov{C}_X^{\lambda}$ to be the quotient double complex
$$
\ov{C}_X^{\lambda} := \ddot{X}/\ker(\ddot{B}).
$$
It is easy to check that
$$
\ov{C}_X^{\lambda} = \frac{P(\ddot{X})\oplus P^{\perp}(\ddot{X})}{\ima(\ddot{B})\oplus P^{\perp}(\ddot{X})} = \frac{P(\ddot{X})}{\ima(\ddot{B})}
$$
and that $\breve{B}$ induces the isomorphism of complexes $\Tot(\ov{C}_X^{\lambda})[1] \simeq \ima(\breve{B})$. So, we have the short exact sequence of double complexes
$$
\begin{tikzcd}[column sep=1cm]
0 \arrow{r} & \overline{C}_X^{\lambda}[0,1] \arrow{r}{i} & {P(\ddot{X})} \arrow{r}{j} & \ov{C}_X^{\lambda} \arrow{r} & 0
\end{tikzcd}
$$
where $j$ is the canonical surjection and $i$ is induced by $\breve{B}$. Taking total complexes we obtain the short exact sequence
$$
\begin{tikzcd}[column sep=1cm]
0 \arrow{r} & \Tot(\overline{C}_X^{\lambda})[1] \arrow{r}{i} & {P(\breve{X})} \arrow{r}{j} & \Tot(\ov{C}_X^{\lambda}) \arrow{r} & 0
\end{tikzcd}
$$

\subsection{A description of $P(\ddot{X})$}\label{sAdescription}
In this subsection we obtain a precise description of the double mixed complex $(P(\ddot{X}), \ddot{b},\ddot{d},\ddot{B})$. The main result is Theorem~\ref{th4.7}. We relegate its proof to Appendix~B.

Recall from Theorem~\ref{th3.4} that $\ov{X}_{vw}$ is the cokernel of $\ide-t\colon X_{vw}\longrightarrow
X_{vw}$. Consider the double complex
$$
(\wt{\mathfrak{X}}{,\hspace{1pt}}\wt{\mathfrak{b}}{,\hspace{1pt}}\wt{\mathfrak{d}}):=
\begin{tikzcd}[column sep=0.8cm,row sep=0.8cm]
&& \vdots  \arrow{d}{\wt{\mathfrak{b}}} & \vdots \arrow{d}{\wt{\mathfrak{b}}}& \vdots \arrow{d}{\wt{\mathfrak{b}}}\\
&\dots \arrow{r}{\wt{\mathfrak{d}}} & \wt{\mathfrak{X}}_{22} \arrow{r}{\wt{\mathfrak{d}}} \arrow{d}{\wt{\mathfrak{b}}} & \wt{\mathfrak{X}}_{21} \arrow{r}{\wt{\mathfrak{d}}} \arrow{d}{\wt{\mathfrak{b}}} & \wt{\mathfrak{X}}_{20} \arrow{d}{\wt{\mathfrak{b}}}\\
&\dots \arrow{r}{\wt{\mathfrak{d}}} & \wt{\mathfrak{X}}_{12} \arrow{r}{\wt{\mathfrak{d}}} \arrow{d}{\wt{\mathfrak{b}}} & \wt{\mathfrak{X}}_{11} \arrow{r}{\wt{\mathfrak{d}}} \arrow{d}{\wt{\mathfrak{b}}} & \wt{\mathfrak{X}}_{10} \arrow{d}{\wt{\mathfrak{b}}}\\
&\dots \arrow{r}{\wt{\mathfrak{d}}} & \wt{\mathfrak{X}}_{02} \arrow{r}{\wt{\mathfrak{d}}} & \wt{\mathfrak{X}}_{01} \arrow{r}{\wt{\mathfrak{d}}} & \wt{\mathfrak{X}}_{00},
\end{tikzcd}
$$
\noindent where $\wt{\mathfrak{X}}_{vw} := \overline{X}_{vw}\oplus \overline{X}_{v,w-1}$ and the boundary maps are defined by
\[
\wt{\mathfrak{b}}(\bx,\byy) := (\overline{b}(\bx),-\overline{b}(\byy))\qquad\text{and}\qquad \wt{\mathfrak{d}}(\bx,\byy) := (\overline{d}(\bx),-\overline{d}(\byy)),
\]
respectively. Let $\mathfrak{p}\colon X_{vw}\longrightarrow \ov{X}_{vw}$ be the map defined by
\[
\mathfrak{p}(\bx) := \frac{1}{w+1}[\bx]\qquad\text{for each very simple tensor $\bx\in X_{vw}$,}
\]
where $[\bx]$ denotes the class of $\bx$ in $\ov{X}_{vw}$. Let $\ov{N}\colon \ov{X}_{vw}\longrightarrow X_{vw}$ be the map induced by $N$. It is easy to check that
\begin{equation}\label{e4}
\ov{N}\xcirc \ov{b} = b\xcirc \ov{N},\qquad\ov{N}\xcirc \ov{d} = d'\xcirc \ov{N}\qquad\text{and}\qquad \mathfrak{p}\xcirc \ov{N} = \ide.
\end{equation}
Let
$$
\wt{\xi}\colon \wt{X}_{vw}\longrightarrow \wt{X}_{v,w-1}
$$
be the maps defined by $\wt{\xi}(\bx,\byy) := (0,\ov{\xi}(\bx))$, where $\ov{\xi}\colon \ov{X}_{vw}\longrightarrow \ov{X}_{v,w-2}$ is given by
$$
\ov{\xi}(\bx) := \begin{cases} \frac{1}{w+1}\mathfrak{p}\xcirc d' \xcirc \sigma'\xcirc  d\xcirc \ov{N}(\bx) &\text{if $w>1$,}\\ 0 &\text{if $w\le 1$,}\end{cases}
$$
where $\sigma'\colon X_{v,w-1}\longrightarrow X_{v,w-1}$ is the morphism introduced in the proof of Lemma~\ref{le3.3}.

\begin{proposition}\label{pr4.5} For $\bx\in\ov{X}_{vw}$ with $w>1$, we have
$$
\ov{\xi}(\bx) = -\frac{1}{w+1} \mathfrak{p}\xcirc \varrho_{n-1} \xcirc \sigma'\xcirc \varrho_n\xcirc \ov{N}(\bx).
$$
\end{proposition}

\begin{proof} In fact
$$
(w+1)\ov{\xi}(\bx) = \mathfrak{p}\xcirc d \xcirc \sigma'\xcirc  d\xcirc \ov{N}(\bx) - \mathfrak{p}\xcirc \varrho_{n-1} \xcirc \sigma'\xcirc  d\xcirc \ov{N}(\bx) = - \mathfrak{p}\xcirc \varrho_{n-1} \xcirc \sigma'\xcirc  d\xcirc \ov{N}(\bx),
$$
since
\begin{align*}
\mathfrak{p}\xcirc d \xcirc \sigma'\xcirc  d\xcirc \ov{N}(\bx) &= \frac{w}{w-1} \ov{d} \xcirc \mathfrak{p}\xcirc \sigma'\xcirc  d\xcirc \ov{N}(\bx)\\
& = \sum_{j=0}^{w-2}\frac{w-j-1}{w-1} \ov{d} \xcirc \mathfrak{p}\xcirc t^j\xcirc d\xcirc \ov{N}(\bx)\\
& = \sum_{j=0}^{w-2} \frac{w-j-1}{w(w-1)} \ov{d} \xcirc \ov{d}\bigl([\ov{N}(\bx)]\bigr)\\
& = 0,
\end{align*}
where $[\ov{N}(\bx)]$ denotes the class  of $\ov{N}(\bx)\in X_{vw}$ in $\ov{X}_{vw}$. But
$$
\mathfrak{p}\xcirc\varrho_{n-1}\xcirc\sigma'\xcirc d\xcirc\ov{N}(\bx) = \mathfrak{p}\xcirc \varrho_{n-1} \xcirc \sigma'\xcirc  d'\xcirc \ov{N}(\bx) + \mathfrak{p}\xcirc\varrho_{n-1}\xcirc\sigma'\xcirc\varrho_n\xcirc\ov{N}(\bx)= \mathfrak{p}\xcirc\varrho_{n-1}\xcirc\sigma'\xcirc \varrho_n\xcirc\ov{N}(\bx),
$$
since
\begin{align*}
\mathfrak{p}\xcirc \varrho_{n-1} \xcirc \sigma'\xcirc  d'\xcirc \ov{N}(\bx) &= \mathfrak{p}\xcirc \varrho_{n-1} \xcirc \sigma'\xcirc  \ov{N}\xcirc \ov{d}(\bx)\\
&= \frac{w-1}{2}\mathfrak{p}\xcirc \varrho_{n-1} \xcirc  \ov{N}\xcirc \ov{d}(\bx)\\
&= \frac{1}{2}\ov{d} \xcirc\ov{d}(\bx)\\
& = 0,
\end{align*}
which finishes the proof.
\end{proof}

\begin{proposition}\label{pr4.6} Let $\bx_0^n\in\ov{X}_{vw}$ with $w>1$ and let $0=i_0<\cdots<i_w\le n$ be the indices such that $x_{i_j}\in M$. Then
$$
\ov{\xi}(\bx_0^n) = \sum_{0\le j<l\le w}(-1)^{(i_l+1)(n+1)+i_j}\frac{(w+2j-2l+1)}{(w-1)w(w+1)} \bigl[\varrho_{i_j} \xcirc \varrho_{i_l}(\bx_o^n)\bigr],
$$
where $\bigl[\varrho_{i_j} \xcirc \varrho_{i_l}(\bx_o^n)\bigr]$ denotes the class of $\varrho_{i_j} \xcirc \varrho_{i_l}(\bx_o^n)\in X_{v,w-2}$ in $\ov{X}_{v,w-2}$.
\end{proposition}

\begin{proof} This follows from Proposition~\ref{pr4.5}. We leave the details to the reader.
\end{proof}

Given a $t$-invariant element $\bx\in X_{vw}$, let
$$
\Upsilon(\bx) :=\begin{cases} (-\ide + \ov{N}\xcirc \mathfrak{p})\xcirc \sigma'\xcirc d(\bx)\in X_{v,w-1} &\text{ if $w>0$,}\\ 0 & \text{ if $w = 0$.}\end{cases}
$$

\begin{theorem}\label{th4.7} Let $\wt{\mathfrak{B}}\colon \wt{X}_{vw}\longrightarrow \wt{X}_{v,w+1}$ be the map defined by $\wt{\mathfrak{B}}(\bx,\byy) := (0,\bx)$. The diagram $(\wt{X},\wt{\mathfrak{b}},\wt{\mathfrak{d}} + \wt{\xi},\wt{\mathfrak{B}})$ is a mixed double complex and the map
\[
\Psi\colon (\wt{X},\wt{\mathfrak{b}},\wt{\mathfrak{d}}+\wt{\xi},\wt{\mathfrak{B}})\longrightarrow (P(\ddot{X}),\ddot{b},\ddot{d},\ddot{B}),
\]
defined  by
$$
\Psi_{vw}(\bx,\byy) := \frac{1}{w+1}\bigl(\ov{N}(\bx),\Upsilon\xcirc \ov{N}(\bx)\bigr) + \bigl(0,\ov{N}(\byy)\bigr),\\
$$
is an isomorphism of double mixed complexes.
\end{theorem}

\begin{proof} See Appendix~B.
\end{proof}

\begin{proposition}\label{prop 3.10} The following assertions hold:

\begin{enumerate}

\smallskip

\item The map
$$
\Gamma\colon \Tot(\ov{X},\ov{b},\ov{d})\longrightarrow \Tot \BC(\wt{X},\wt{\mathfrak{b}}+\wt{\mathfrak{d}}+\wt{\xi}, \wt{\mathfrak{B}}),
$$
defined by
$$
\Gamma_{vw}(\bx):= \bigl((\bx,0),(-\ov{\xi}(\bx),0),(\ov{\xi}^2(\bx),0),(-\ov{\xi}^3(\bx),0),(\ov{\xi}^4(\bx),0),\dots\bigr),
$$
is a morphism of complexes.

\smallskip

\item The map
$$
\Pi\colon \Tot \BC(\wt{X},\wt{\mathfrak{b}}+\wt{\mathfrak{d}}+\wt{\xi}, \wt{\mathfrak{B}})\longrightarrow \Tot(\ov{X},\ov{b},\ov{d}),
$$
defined by
$$
\Pi_{vw}\bigl((\bx_0,\byy_0),(\bx_1,\byy_1),(\bx_2,\byy_2),(\bx_3,\byy_3),(\bx_4,\byy_4),\dots\bigr):= \bx_0,
$$
is a morphism of complexes.

\smallskip

\item $\Pi\xcirc \Gamma = \ide$ and $\Gamma\xcirc \Pi$ is homotopic to the identity map. A homotopy is the family of maps
$$
\Xi\colon \Tot\BC(\wt{X},\wt{\mathfrak{b}}+\wt{\mathfrak{d}}+\wt{\xi},\wt{\mathfrak{B}})_w\longrightarrow \Tot\BC(\wt{X},\wt{\mathfrak{b}}+ \wt{\mathfrak{d}} + \wt{\xi},\wt{\mathfrak{B}})_{w+1},
$$
defined by
$$
\Xi_{vw}(\bz):= \bigl((0,0),(-\ov{\byy}_0,0),(-\ov{\byy}_1,0),(-\ov{\byy}_2,0),(-\ov{\byy}_3,0), \dots\bigr),
$$
where
$$
\qquad \bz := \bigl((\bx_0,\byy_0),(\bx_1,\byy_1),(\bx_2,\byy_2),(\bx_3,\byy_3),\dots\bigr)\qquad\text{and}\qquad \ov{\byy}_i:= \sum_{j=0}^i (-1)^j \ov{\xi}^j(\byy_{i-j}).
$$
\end{enumerate}
\end{proposition}

\begin{proof} Consider the following special deformation retract
$$
\begin{tikzpicture}[label distance=5mm]
\draw (0,0) node {$\Tot(\ov{X},\ov{b},\ov{d})$};\draw (4.5,0) node {$\Tot\BC(\wt{X},\wt{\mathfrak{d}}+\wt{\mathfrak{d}},\wt{\mathfrak{B}})$,};\draw[<-] (1,0.12) -- node[above=-2pt,font=\scriptsize] {$\Pi'$}(2.8,0.12);\draw[->] (1,-0.12) -- node[below=-2pt,font=\scriptsize] {$\Gamma'$}(2.8,-0.12);\draw (8,0) node {$\Xi'$,};
\end{tikzpicture}
$$
where $\Gamma'$, $\Pi'$ and $\Xi'$ are given by
\begin{align*}
&\Gamma'_{vw}(\bx):= \bigl((\bx,0),(0,0),(0,0),(0,0),(0,0),\dots\bigr),\\
&\Pi'_{vw}\bigl((\bx_0,\byy_0),(\bx_1,\byy_1),(\bx_2,\byy_2),(\bx_3,\byy_3),(\bx_4,\byy_4),\dots\bigr):= \bx_0
\shortintertext{and}
&\Xi'_{vw}\bigl(\bigl((\bx_0,\byy_0),(\bx_1,\byy_1),(\bx_2,\byy_2),\dots\bigr):= \bigl((0,0),(-\byy_0,0), (-\byy_1,0),(-\byy_2,0),\dots\bigr),
\end{align*}
endowed with the perturbation
$$
\Omega'\colon\bigoplus_{i\ge 0}\wt{X}_{*,*-i}\longrightarrow \bigoplus_{i\ge 0}\wt{X}_{*,*-1-i},
$$
given by
$$
\Omega'_{vw}\bigl((\bx_0,\byy_0),(\bx_1,\byy_1),(\bx_2,\byy_2),\dots\bigr):= \bigl((0,\ov{\xi}(\bx_0)), (0,\ov{\xi}(\bx_1)),(0,\ov{\xi}(\bx_2)), \dots\bigr).
$$
The proposition follows applying the perturbation lemma to this datum.
\end{proof}

Let $i\colon \ov{X}_{v,w-1}\longrightarrow \wt{X}_{vw}$ and $\pi\colon \wt{X}_{vw}\longrightarrow \ov{X}_{vw}$ be the canonical maps. The short exact sequence of double complexes
\begin{equation}
\begin{tikzcd}[column sep=1cm]
0 \arrow{r} & (\ov{X}_{*,*-1},\ov{b},\ov{d}) \arrow{r}{i} & {(\wt{X}_{**},\wt{\mathfrak{b}}, \wt{\mathfrak{d}}+\wt{\xi})} \arrow{r}{\pi} & (\ov{X}_{**},\ov{b},\ov{d}) \arrow{r} & 0
\end{tikzcd} \label{eq16}
\end{equation}
splits in each level via the maps $s\colon \ov{X}_{vw} \longrightarrow \wt{X}_{vw}$ and $r\colon \wt{X}_{vw} \longrightarrow \ov{X}_{v,w-1}$, given by $s(\bx) := (\bx,0)$ and $r(\bx,\byy) := \byy$. From this it follows immediately that the connection map of the homology long exact sequence associated with~\eqref{eq16} is induced by the morphism of double complexes
$$
\ov{\xi}\colon (\ov{X}_{**},\ov{b},\ov{d})\longrightarrow (\ov{X}_{*,*-2},\ov{b},\ov{d}).
$$

\begin{proposition}\label{pr4.8} The maps
\begin{align*}
& S_n\colon \HC_n(E,M)\longrightarrow \HC_{n-2}(E,M),\\
& B_n\colon \HC_n(E,M)\longrightarrow \HH_{n+1}(E,M)
\shortintertext{and}
& i_n\colon \HH_n(E,M)\longrightarrow \HC_n(E,M),
\end{align*}
are induced by $-\ov{\xi}$, $i$ and $\pi$, respectively.
\end{proposition}

\begin{proof} This follows by a direct computation using Proposition~\ref{prop 3.10}. We leave the details to the reader.
\end{proof}

\begin{proposition}\label{pr4.9} Let $\wt{P}\colon \ddot{X}\to \wt{X}$ be the map $\Psi^{-1}\xcirc P$. We have
\begin{align}
\wt{P}(0,\byy) & = \frac{1}{w} (0,[\byy])\qquad \text{for $(0,\byy)\in \ddot{X}_{vw}$ with $w>0$.}\label{et1}
\shortintertext{and}
\wt{P}(\bx,0) & = \begin{cases} ([\bx],0) &\text{if $(\bx,0) \in \ddot{X}_{n0}$,} \\ ([\bx],0) + \displaystyle{\sum_{i=0}^w \frac{2i-w}{2w(w+1)} \bigl(0,[\varrho_n\xcirc t^i(\bx)]\bigr)} &\text{if $(\bx,0) \in \ddot{X}_{vw}$ with $w>0$.}\end{cases}\label{et2}
\end{align}
\end{proposition}

\begin{proof} See Appendix~B.
\end{proof}

\begin{remark} Let $M$ be an associative nonunital $k$-algebra. Let $E$ be the augmented algebra obtained adjoin the unit of $k$ to $M$. In \cite{C-Q2} the authors introduced and studied the the harmonic decomposition of the mixed complex of $E$. Applying Theorem~\ref{th4.7} to $E$ we obtain a complex isomorphic to the harmonic part of of this decomposition.
\end{remark}

\setcounter{secnumdepth}{0}
\setcounter{section}{1}
\setcounter{theorem}{0}
\section{Appendix A}
\renewcommand\thesection{\Alph{section}}
\renewcommand\sectionmark[1]{}
This appendix is devoted to prove Theorem~\ref{th3.2}.

\smallskip

Let $E := A\ltimes_{\!\smalltriangledown} M$ be a cleft extension and let $K$ be a $k$-subalgebra of $A$. The $w$-th column $(\hat{\mathfrak{X}}_{*w},\hat{\mathfrak{b}})$ of the double complex $(\hat{\mathfrak{X}}, \hat{\mathfrak{b}},\hat{\mathfrak{d}})$ introduced at the beginning of Section~\ref{relative} is the total complex of the double complex
$$
(\mathfrak{X}_w{,\hspace{1pt}}\mathfrak{b}{,\hspace{1pt}}\alpha) :=
\begin{tikzcd}[column sep=1.2cm,row sep=0.8cm]
&  \vdots \arrow{d}{\mathfrak{b}^0}& \vdots \arrow{d}{\mathfrak{b}^1}\\
& \mathfrak{X}^0_{2w} \arrow{d}{\mathfrak{b}^0} & \arrow{l}[swap]{\alpha}  \mathfrak{X}^1_{2w} \arrow{d}{\mathfrak{b}^1}\\
& \mathfrak{X}^0_{1w} \arrow{d}{\mathfrak{b}^0} & \arrow{l}[swap]{\alpha} \mathfrak{X}^1_{1w} \arrow{d}{\mathfrak{b}^1}\\
&  \mathfrak{X}^0_{0w} & \arrow{l}[swap]{\alpha} \mathfrak{X}^1_{0w},
\end{tikzcd}
$$
where $\mathfrak{X}^0_{vw} := M\ot B_w^n\ot$ and $\mathfrak{X}^1_{v-1,w} := A\ot B_{w+1}^n\ot$.

\smallskip

Lemmas~\eqref{leA.1} and \eqref{leA.2} below was proven in \cite[Appendix~A]{G-G}.

\begin{lemma}\label{leA.1} Let $\theta^1\colon (\mathfrak{X}^1_{*w},\mathfrak{b}^1)\longrightarrow (X_{*w},-b)$ and $\vartheta^1 \colon (X_{*w},-b)\longrightarrow (\mathfrak{X}^1_{*w}, \mathfrak{b}^1)$ be the morphisms of complexes given by
\[
\theta^1(\byy_0^{n+1}) := \mu_0^M(\byy_0^{n+1})\qquad\text{and}\qquad \vartheta^1(\byy_0^n) := \sum_{l=0}^{n-i(\byy_0^n)} 1\ot \mathfrak{t}^l(\byy_0^n),
\]
where $\mathfrak{t}(y_0\ot\cdots\ot y_n) := (-1)^n y_n\ot y_0\ot\cdots\ot y_{n-1}$. Then, $\theta^1\xcirc \vartheta^1 = \ide$ and $\vartheta^1\xcirc \theta^1$ is homotopic to $\ide$. A homotopy is the family of maps $\ep_{*w}\colon \mathfrak{X}^1_{*-1,w}\longrightarrow \mathfrak{X}^1_{*w}$, defined by
\[
\ep(\byy_0^n) := -\sum_{l=0}^{n-i(\byy_0^n)} 1\ot \mathfrak{t}^l(\byy_0^n).
\]
\end{lemma}

\begin{lemma}\label{leA.2} For $w\ge 0$, let $\tau_0^1(\mathcal{X}_{*w})$ be the double diagram with two columns
$$
\begin{tikzcd}[column sep=1.2cm,row sep=0.8cm]
(X_{*w},b)& (X_{*w},-b).  \arrow{l}[swap]{\ide-t}
\end{tikzcd}
$$
The following assertions hold:

\begin{enumerate}

\smallskip

\item $\tau_0^1(\mathcal{X}_{*w})$ is a double complex.

\smallskip

\item The map $\vartheta\colon \tau_0^1(\mathcal{X}_{*w}) \longrightarrow (\mathfrak{X}_{*w},\mathfrak{b},\alpha)$, defined by $\vartheta :=(\vartheta^0, \vartheta^1)$, where $\vartheta^0$ is the identity map and $\vartheta^1$ is as in Lemma~\ref{leA.1}, is a morphism of double complexes.

\smallskip

\item The map $\hat{\theta}\colon (\hat{\mathfrak{X}}_{*w},\hat{\mathfrak{b}})\longrightarrow \Tot(\tau_0^1(\mathcal{X}_{*w}))$, defined by
\[
\quad\qquad\hat{\theta}(\bx_0^n,\byy_0^n) := (\bx_0^n + t(\byy_0^n),\mu_0^M(\byy_0^n)),
\]
is a morphism of complexes.

\smallskip

\item Let $\hat{\vartheta}\colon \Tot(\tau_0^1(\mathcal{X}_{*w})) \longrightarrow (\hat{\mathfrak{X}}_{*w},\hat{\mathfrak{b}})$ be the map induced by $\vartheta$. It is true that $\hat{\theta}\xcirc \hat{\vartheta} = \ide$ and that $\hat{\vartheta}\xcirc \hat{\theta}$ is homotopic to the identity map. A homotopy is the family of maps
\[
\quad\qquad \hat{\ep}_{v+1,w}\colon \mathfrak{X}^0_{vw}\oplus \mathfrak{X}^1_{v-1,w}\longrightarrow \mathfrak{X}^0_{v+1,w} \oplus \mathfrak{X}^1_{vw}\qquad (v\ge 0),
\]
defined by $\hat{\ep}(\bx_0^n,\byy_0^n):= (0,\ep(\byy_0^n))$, where $\ep$ is the homotopy introduced in Lemma~\ref{leA.1}.
\end{enumerate}
\end{lemma}

Recall from the proof of Lemma~\ref{th3.1} that $(\breve{X},\breve{b}):= \Tot(\hat{X},\hat{b},\hat{d})$ and from the beginning of this section that $(\breve{\mathfrak{X}},\breve{\mathfrak{b}}) := \Tot(\hat{\mathfrak{X}}, \hat{\mathfrak{b}},\hat{\mathfrak{d}})$. The first item of the following lemma is part of item~(1) of Theorem~\ref{th3.2}.

\begin{lemma}\label{leA.3} The following assertions hold:

\smallskip

\begin{enumerate}

\item The diagram $(\hat{X},\hat{b},\hat{d})$, introduced above Theorem~\ref{th3.2}, is a double complex.

\smallskip

\item The map $\hat{\vartheta}\colon (\hat{X},\hat{b},\hat{d})\longrightarrow (\hat{\mathfrak{X}},\hat{\mathfrak{b}}, \hat{\mathfrak{d}})$ is a morphism of double complexes.

\smallskip

\item The map $\hat{\theta}\colon (\hat{\mathfrak{X}},\hat{\mathfrak{b}},\hat{\mathfrak{d}})\longrightarrow (\hat{X},\hat{b},\hat{d})$ is a morphism of double complexes.

\smallskip

\item Let $\breve{\vartheta}\colon (\breve{X},\breve{b}) \longrightarrow (\breve{\mathfrak{X}}, \breve{\mathfrak{b}})$ and $\breve{\theta}\colon (\breve{\mathfrak{X}},\breve{\mathfrak{b}}) \longrightarrow (\breve{X},\breve{b})$ be the maps induced by $\hat{\vartheta}$ and $\hat{\theta}$, respectively.  It is true that $\breve{\theta} \xcirc \breve{\vartheta} = \ide$ and that $\breve{\vartheta}\xcirc \breve{\theta}$ is homotopic to the identity map. A homotopy is the family of maps $\breve{\ep}_{n+1}\colon \breve{\mathfrak{X}}_n \longrightarrow \breve{\mathfrak{X}}_{n+1}$, defined by $\breve{\ep}_{n+1} := \bigoplus_{w=0}^n \hat{\ep}_{n+1-w,w}$, where $\hat{\ep}_{n+1-w,w}$ is as in Lemma~\ref{leA.2}.
\end{enumerate}
\end{lemma}

\begin{proof} By Lemma~\ref{leA.2} we have the following special deformation retract
$$
\begin{tikzpicture}[label distance=5mm]
\draw (0,-0.1) node {$\displaystyle{\bigoplus_{w\ge 0}} (\hat{X}_{*w},\hat{b})$};\draw (4.5,-0.1) node {$\displaystyle{\bigoplus_{w\ge 0}} (\hat{\mathfrak{X}}_{*w}, \hat{\mathfrak{b}})$,};\draw[<-] (1,0.12) -- node[above=-2pt,font=\scriptsize] {$\hat{\theta}$}(3.4,0.12);\draw[->] (1,-0.12) -- node[below=-2pt,font=\scriptsize] {$\hat{\vartheta}$}(3.4,-0.12);\draw (7,0) node {$\hat{\ep}$,};
\end{tikzpicture}
$$
where $\hat{\vartheta} := \bigoplus_{w\ge 0} \hat{\vartheta}_{*w}$, $\hat{\theta} := \bigoplus_{w\ge 0} \hat{\theta}_{*w}$ and $\hat{\ep} := \bigoplus_{w\ge 0} \hat{\ep}_{*w}$. Consider the perturbation $\hat{\mathfrak{d}} := \bigoplus_{w\ge 0} \hat{\mathfrak{d}}_{*w}$. Applying the perturbation lemma to this datum, we obtain a special deformation retract
$$
\begin{tikzpicture}[label distance=5mm]
\draw (0,0) node {$(\breve{X},\ov{b})$};\draw (4,0) node {$(\breve{\mathfrak{X}},\breve{\mathfrak{b}})$,};\draw[<-] (0.6,0.12) -- node[above=-2pt,font=\scriptsize] {$\overline{\theta}$}(3.4,0.12);\draw[->] (0.6,-0.12) -- node[below=-2pt,font=\scriptsize] {$\overline{\vartheta}$}(3.4,-0.12);\draw (7,0) node {$\ov{\ep}_{*+1}\colon \breve{\mathfrak{X}}_*\longrightarrow \breve{\mathfrak{X}}_{*+1}$.};
\end{tikzpicture}
$$
To finish the proof it remains to check that $\ov{b} = \breve{b}$, $\ov{\vartheta} = \breve{\vartheta}$, $\ov{\theta} = \breve{\theta}$ and $\ov{\ep} = \breve{\ep}$, for which it suffices to check that
$$
\hat{\theta}\xcirc \hat{\mathfrak{d}}\xcirc \hat{\vartheta} = \hat{d},\qquad \hat{\ep}\xcirc \hat{\mathfrak{d}}\xcirc \hat{\vartheta}= 0,\qquad \hat{\theta}\xcirc \hat{\mathfrak{d}}\xcirc\hat{\ep}= 0\qquad\text{and}\qquad \hat{\ep}\xcirc \hat{\mathfrak{d}}\xcirc \hat{\ep} = 0,
$$
which follows by a direct computation.
\end{proof}

\begin{lemma}\label{leA.4} Let $\hat{B}$ be as in Theorem~\ref{th3.2}. The following assertions hold:
\begin{enumerate}

\smallskip

\item $(\hat{X},\hat{b},\hat{d},\hat{B})$ is a double mixed complex.

\smallskip

\item The maps
\[\qquad\quad
\hat{\vartheta}\colon (\hat{X},\hat{b},\hat{d},\hat{B}) \longrightarrow (\hat{\mathfrak{X}},\hat{\mathfrak{b}}, \hat{\mathfrak{d}},\hat{\mathfrak{B}}) \qquad \text{and}\qquad \hat{\theta}\colon (\hat{\mathfrak{X}}, \hat{\mathfrak{b}},\hat{\mathfrak{d}},\hat{\mathfrak{B}}) \longrightarrow (\hat{X},\hat{b},\hat{d},\hat{B})
\]
are morphisms of double mixed complexes.

\smallskip

\item For $W$ equals $C$, $N$ and $P$, let
\begin{align*}
&\check{\vartheta}\colon \Tot \BW(\breve{X},\breve{b},\breve{B})\longrightarrow \Tot\BW(\breve{\mathfrak{X}},\breve{\mathfrak{b}},\breve{\mathfrak{B}})
\shortintertext{and}
&\check{\theta}\colon \Tot\BW(\breve{\mathfrak{X}},\breve{\mathfrak{b}},\breve{\mathfrak{B}})\longrightarrow \Tot \BW(\breve{X},\breve{b},\breve{B})
\end{align*}
be the maps induced by the morphisms $\breve{\vartheta}$ and $\breve{\theta}$ introduced in item~4) of Lemma~\ref{leA.3}. It is true that $\check{\theta} \xcirc \check{\vartheta} = \ide$ and $\check{\vartheta}\xcirc \check{\theta}$ is homotopic to the identity map. A homotopy is the family of maps
$$
\check{\ep}\colon \Tot\BW(\breve{\mathfrak{X}},\breve{\mathfrak{b}},\breve{\mathfrak{B}})_* \longrightarrow \Tot\BW(\breve{\mathfrak{X}},\breve{\mathfrak{b}},\breve{\mathfrak{B}})_{*+1},
$$
defined by applying $\breve{\ep}$ on each component.
\end{enumerate}
\end{lemma}

\begin{proof} 1)\enspace From the fact that $\hat{b}\xcirc \hat{b} = 0$ it follows easily that $t\xcirc b = b\xcirc t$. Thus we obtain that $b\xcirc N = N\xcirc b$, which implies that $\hat{b}\xcirc \hat{B} + \hat{B}\xcirc \hat{b} = 0$. To prove that $\hat{d}\xcirc \hat{B} + \hat{B}\xcirc \hat{d} = 0$ we must check that $d'\xcirc N = N\xcirc d$. Let $\bx_0^n\in X_{vw}$ be a very simple tensor. Let $0=i_0<i_1<\cdots<i_w\le n$ be the indices such that $x_{i_j}\in M$ and let $i_{w+1} = n+1$. A direct computation shows that for $0\le l\le n$ and $0\le j<w$,
$$
t^j\xcirc \varrho_l(\bx_0^n) = \begin{cases} \varrho_{l+n+1-i_{w+1-j}}\xcirc t^j(\bx_0^n)&\text{if $l<i_{w+1-j}-1$,}\\ \varrho_{l-i_{w-j}}\xcirc t^{j+1}(\bx_0^n)&\text{if $l\ge i_{w+1-j}-1$.}
\end{cases}
$$
Hence
$$
N\xcirc d(\bx_0^n) = \sum_{j=0}^{w-1} \sum_{l=0}^n t^j\xcirc \varrho_l(\bx_0^n) = \sum_{l=0}^{n-1}\sum_{j=0}^w \varrho_l\xcirc t^j(\bx_0^n) = d'\xcirc N(\bx_0^n),
$$
as we want.

\smallskip

\noindent 2) and 3)\enspace From Lemma~\ref{leA.3} we get a special deformation retract between the total complexes of the double complexes $\BC(\breve{X},\breve{b},0)$ and $\BC(\breve{\mathfrak{X}},\breve{\mathfrak{b}},0)$. Consider the perturbation induced by $\breve{\mathfrak{B}}$. The result it follows by applying the perturbation lemma to this setting, and using that $\breve{B} = \breve{\theta}\xcirc \breve{\mathfrak{B}}\xcirc \breve{\vartheta}$, $\breve{\mathfrak{B}}\xcirc \breve{\ep} = 0$ and $\breve{\ep}\xcirc \breve{\mathfrak{B}} = 0$.
\end{proof}

\noindent{\bf Proof of Theorem~\ref{th3.2}.}\enspace This follows immediately from Lemma~\ref{leA.4}.\qed

\setcounter{section}{2}
\setcounter{theorem}{0}
\section{Appendix B}
Recall from the discussion above Proposition~\ref{pr4.5}, that for each $t$-invariant element $\bx\in X_{vw}$,
$$
\Upsilon(\bx) :=\begin{cases} (-\ide + \ov{N}\xcirc \mathfrak{p})\xcirc \sigma'\xcirc d(\bx)\in X_{v,w-1} &\text{ if $w>0$,}\\ 0 & \text{ if $w = 0$.}\end{cases}
$$
We claim that $\Upsilon(\bx)$ is univocally determined by the following properties:
\[
\Upsilon(\bx)\in \ker(\mathfrak{p})\qquad\text{and}\qquad (\bx,\Upsilon(\bx))\in P(\ddot{X}).
\]
First note that $\mathfrak{p}\xcirc \Upsilon = 0$, since $\mathfrak{p}$ is a retraction of $\ov{N}$. By Remark~\ref{info sobre P(X)}, in order to prove that $(\bx,\Upsilon(\bx))\in P(\ddot{X}_{vw})$ it suffices to see that $d(\bx)+(\ide-t)\xcirc \Upsilon(\bx)$ is $t$-invariant. This is evident for $w=0$ and it is true for $w>0$, since
\begin{equation}
d(\bx)+(\ide-t)\xcirc \Upsilon(\bx) = d(\bx)+(t-\ide)\xcirc \sigma'\xcirc d(\bx) = \frac{1}{w} N\xcirc d(\bx),\label{*eq17}
\end{equation}
where the first equality follows from the fact that $t\circ \ov{N} = \ov{N}$ and the second one, from equality~\eqref{*eq5}. Conversely, assume that $(\bx,\byy)\in P(\ddot{X}_{vw})$ and let $[\byy - \Upsilon(\bx)]$ denote the class of $\byy - \Upsilon(\bx)$ in $\ov{X}_{v,w+1}$. Again by Remark~\ref{info sobre P(X)}, we know that $\byy - \Upsilon(\bx)$ is $t$-invariant, and so, if $\byy\in \ker(\mathfrak{p})$, then
$$
[\byy - \Upsilon(\bx)] = \mathfrak{p}\xcirc \ov{N}\bigl([\byy - \Upsilon(\bx)]\bigr) = \mathfrak{p}\xcirc N\bigl(\byy - \Upsilon(\bx)\bigr) = w \mathfrak{p} \bigl(\byy - \Upsilon(\bx)\bigr) = 0,
$$
which implies that $\byy = \Upsilon(\bx)$ (note that if $w=0$, then $\byy = \Upsilon(\bx) = 0$). Recall also from Remark~\ref{info sobre P(X)}, that ${}^t\!P(\ddot{X}_{vw})$ is the set of all elements of the shape $(0,\bx)$ that belongs to $P(\ddot{X}_{vw})$, and that
$$
{}^t\!P(\ddot{X}_{vw}) = \{(0,\bx): \bx\in X_{v,w-1}\text{ and } t(\bx) = \bx\} = P(\ddot{X}_{vw})\cap \acute{X}_{v,w-1}.
$$
Let
$$
{}^e\! P(\ddot{X}_{vw}) := \{(\bx,\Upsilon(\bx)) \in \ddot{X}_{vw}: \bx \text{ is $t$-invariant}\}.
$$
Clearly, $P(\ddot{X}) = {}^e\! P(\ddot{X}) \oplus {}^t\!P(\ddot{X})$. We assert that $\ddot{b}({}^e P(\ddot{X}_{vw})) \subseteq {}^e\! P(\ddot{X}_{v-1,w})$, for each $v>0$. In order to prove this we will need the following result.

\begin{lemma}\label{leB.1} For $v>0$, we have $\ov{b}\xcirc \mathfrak{p}=\mathfrak{p}\xcirc b$.
\end{lemma}

\begin{proof} This is trivial.
\end{proof}

\begin{proposition}\label{prB.2} Assume that $v>0$ and let $\bx\in X_{vw}$ be a $t$-invariant element. We
have:
\[
\ddot{b}(\bx,\Upsilon(\bx)) = (b(\bx),\Upsilon\xcirc b(\bx)).
\]
\end{proposition}

\begin{proof} By definition
$$
\ddot{b}(\bx,\Upsilon(\bx)) = \bigl(b(\bx),-b(\Upsilon(\bx))\bigr).
$$
Since $b(\bx)$ is $t$-invariant, in order to finish the proof it suffices to check that
\[
\bigl(b(\bx),-b(\Upsilon(\bx))\bigr)\in P(\ddot{X})\qquad\text{and}\qquad \mathfrak{p}(b(\Upsilon(\bx)) = 0.
\]
The first fact follows immediately from the fact that $P(\ddot{X})$ is a subcomplex of $(\ddot{X},\ddot{b},\ddot{d})$ and the second one follows easily from Lemma~\ref{leB.1}.
\end{proof}

For each $w>0$, let
$$
{}^e d\colon {}^e\! P(\ddot{X}_{vw})\to {}^e\! P(\ddot{X}_{v,w-1})\qquad\text{and}\qquad {}^e\xi\colon {}^e\! P(\ddot{X}_{vw})\to {}^t\! P(\ddot{X}_{v,w-1})
$$
be the maps defined by
$$
{}^e d(\bx,\Upsilon(\bx)) + {}^e\xi(\bx,\Upsilon(\bx)) := \ddot{d}(\bx,\Upsilon(\bx)).
$$
We now want to compute these maps. To carry out this task we will need Proposition~\ref{prB.3} below. First note that if $\bx\in X_{vw}$ is a $t$-invariant element, then $d'(\bx)$ also is, since
\begin{equation}
N\xcirc d(\bx) = d'\xcirc N(\bx) = (w+1)d'(\bx).\label{*eq18}
\end{equation}
 For $w\ge 1$, let $\xi\colon {}^t\! X_{vw}\to X_{v,w-2}$ be the map defined by
\[
\xi(\bx) := - \frac{w+1}{w} \Upsilon\xcirc d'(\bx) - d'\xcirc \Upsilon(\bx),
\]
where ${}^t\! X_{vw}$ denotes the set of $t$-invariant elements of $X_{vw}$.

\begin{proposition}\label{prB.3} Assume that $w>0$ and let $\bx\in X_{vw}$ be a $t$-invariant element. Then,
\[
{}^e d(\bx,\Upsilon(\bx)) = \frac{w+1}{w} \bigl(d'(\bx), \Upsilon\xcirc d'(\bx)\bigr)\qquad \text{and}\qquad {}^e \xi(\bx,\Upsilon(\bx)) = (0,\xi(\bx)).
\]
In particular $\xi(\bx)$ is $t$-invariant.
\end{proposition}

\begin{proof} Since $d'(\bx)$ is $t$-invariant,
\begin{equation}
\bigl(d'(\bx),\Upsilon\xcirc d'(\bx)\bigr) \in {}^e\! P(\ddot{X}_{v,w-1}).\label{*eq19}
\end{equation}
Moreover, by equalities~\eqref{*eq17} and~\eqref{*eq18},
\begin{align*}
\ddot{d}(\bx,\Upsilon(\bx)) & = \left(\frac{1}{w} N\xcirc d(\bx),-d'\xcirc \Upsilon(\bx)\right)\\
& = \left(\frac{w+1}{w} d'(\bx),-d'\xcirc \Upsilon(\bx)\right)\\
& = \frac{w+1}{w} \bigl(d'(\bx), \Upsilon\xcirc d'(\bx)\bigr) - \left(0, \frac{w+1}{w} \Upsilon\xcirc d'(\bx) + d'\xcirc \Upsilon(\bx)\right)\\
& = \frac{w+1}{w} \bigl(d'(\bx),\Upsilon\xcirc d'(\bx)\bigr) + (0,\xi(\bx)).
\end{align*}
Consequently in order to finish the proof it suffices to note that
$$
\frac{w+1}{w} \bigl(d'(\bx),\Upsilon\xcirc d'(\bx)\bigr)\in {}^e\! P(\ddot{X}_{v,w-1})\qquad\text{and}\qquad (0,\xi(\bx))\in P(\ddot{X}),
$$
since $P(\ddot{X})$ is a subcomplex of $(\ddot{X},\ddot{b},\ddot{d})$.
\end{proof}

In Proposition~\ref{prB.5} below we will obtain another expression for $\xi$. In order to do this we will need the following result.

\begin{lemma}\label{leB.4} Assume that $w>0$ and let $\bx\in X_{vw}$ be a $t$-invariant element. We have
$$
\ov{N}\xcirc \mathfrak{p}\xcirc \sigma'\xcirc d(\bx) = \frac{(w-1)(w+1)}{2w} d'(\bx).
$$
\end{lemma}

\begin{proof} In fact
\begin{align*}
\ov{N}\xcirc \mathfrak{p}\xcirc \sigma'\xcirc d (\bx) &= \sum_{j=0}^{w-2} \frac{w-j-1}{w} \ov{N}\xcirc \mathfrak{p} \xcirc t^j\xcirc d (\bx)\\
& = \sum_{j=0}^{w-2} \frac{w-j-1}{w^2} \ov{N}([d(\bx)])\\
& = \frac{w-1}{2w} N\xcirc d(\bx)\\
& = \frac{(w-1)(w+1)}{2w} d'(\bx),
\end{align*}
where $[d(\bx)]$ denotes the class of $d(\bx)$ in $\ov{X}_{v,w-1}$ and the last equality follows from~\eqref{*eq18}.
\end{proof}

\begin{proposition}\label{prB.5} Assume that $w>1$ and let $\bx\in X_{vw}$ be a $t$-invariant element. We have
$$
\xi(\bx) = \frac{1}{w-1} N \xcirc d' \xcirc \sigma'\xcirc  d(\bx).
$$
\end{proposition}

\begin{proof} In fact, by Lemma~\ref{leB.4}
$$
d' \xcirc \ov{N}\xcirc \mathfrak{p}\xcirc \sigma'\xcirc d(\bx) = \ov{N}\xcirc \mathfrak{p}\xcirc \sigma'\xcirc d\xcirc d'(\bx)  = 0.
$$
Hence
$$
\xi(\bx) = - \frac{w+1}{w} \Upsilon\xcirc d'(\bx) - d'\xcirc \Upsilon(\bx) =  \frac{w+1}{w} \sigma'\xcirc d \xcirc d'(\bx) + d' \xcirc \sigma'\xcirc  d(\bx).
$$
Combining this with the fact that $\xi(\bx)$ is $t$-invariant, we obtain
\begin{align*}
\xi(\bx) & = \frac{1}{w-1} N\xcirc \xi(\bx)\\
& = \ov{N}\xcirc \mathfrak{p}\xcirc \xi(\bx)\\
& = \frac{w+1}{w} \ov{N}\xcirc \mathfrak{p}\xcirc  \sigma'\xcirc d \xcirc d'(\bx) + \ov{N}\xcirc \mathfrak{p}\xcirc d' \xcirc \sigma'\xcirc d(\bx)\\
& = \ov{N}\xcirc \mathfrak{p}\xcirc d' \xcirc \sigma'\xcirc d(\bx)\\
& = \frac{1}{w-1} N \xcirc d' \xcirc \sigma'\xcirc d(\bx),
\end{align*}
where the forth equality follows from Lemma~\ref{leB.4}.
\end{proof}

\noindent{\bf Proof of Theorem~\ref{th4.7}.}\enspace We claim that $\xi\xcirc \ov{N} = (w+1)\ov{N}\xcirc \ov{\xi}$ on $\ov{X}_{vw}$. This is obvious if $w\le 1$, because both terms in this equality are zero. Assume that $w>1$. Then, by Proposition~\ref{prB.5},
$$
\xi\xcirc \ov{N} =  \frac{1}{w-1}N \xcirc d' \xcirc \sigma'\xcirc  d\xcirc \ov{N} = \ov{N} \xcirc \mathfrak{p}\xcirc d' \xcirc \sigma'\xcirc  d\xcirc \ov{N} = (w+1)\ov{N}\xcirc \ov{\xi}.
$$
Combining this with Proposition~\ref{prB.3} and the second equality in~\eqref{e4}, we obtain that
\begin{align*}
\ddot{d} \xcirc \Psi_{vw}(\bx,\byy)  & = \ddot{d}\left(\frac{1}{w+1} \left(\ov{N}(\bx),\Upsilon\xcirc \ov{N}(\bx)\right) +(0,\ov{N}(\byy)) \right)\\
& =\frac{1}{w+1} {}^e d \left(\ov{N}(\bx),\Upsilon\xcirc \ov{N}(\bx)\right) + \frac{1}{w+1} {}^e \xi \left(\ov{N}(\bx),\Upsilon\xcirc \ov{N}(\bx)\right) + (0,-d'\xcirc \ov{N}(\byy))\\
& = \frac{1}{w} \left(d'\xcirc \ov{N}(\bx),\Upsilon\xcirc d'\xcirc \ov{N}(\bx)\right) + \frac{1}{w+1} \left(0,\xi\xcirc \ov{N}(\bx)\right)-(0,d'\xcirc \ov{N}(\byy))\\
& = \frac{1}{w}\left(\ov{N}\xcirc\ov{d}(\bx),\Upsilon\xcirc\ov{N}\xcirc \ov{d}(\bx)\right)+ (0,\ov{N}\xcirc\ov{\xi}(\bx)) -(0,\ov{N}\xcirc \ov{d}(\byy))\\
& = \Psi_{v,w-1} \left((\ov{d}(\bx),-\ov{d}(\byy)) + (0,\ov{\xi}(\bx))\right)\\
& = \Psi_{v,w-1}\xcirc (\wt{\mathfrak{d}} +\wt{\xi})(\bx,\byy),
\end{align*}
and using Proposition~\ref{prB.2} and the first equality in~\eqref{e4}, that
\begin{align*}
\ddot{b} \xcirc \Psi_{vw}(\bx,\byy) & = \ddot{b}\left(\frac{1}{w+1} \left(\ov{N}(\bx),\Upsilon\xcirc \ov{N}(\bx)\right) + (0,\ov{N}(\byy)) \right)\\
& = \frac{1}{w+1}(b\xcirc \ov{N}(\bx),\Upsilon\xcirc b\xcirc \ov{N}(\bx)) - (0,b\xcirc \ov{N}(\byy))\\
& = \frac{1}{w+1}(\ov{N}\xcirc \ov{b}(\bx),\Upsilon\xcirc \ov{N}\xcirc \ov{b}(\bx)) - (0,\ov{N}\xcirc \ov{b}(\byy))\\
& = \Psi_{v-1,w}(\ov{b}(\bx),-\ov{b}(\byy))\\
& = \Psi_{v-1,w}\xcirc \wt{\mathfrak{b}}(\bx,\byy).
\end{align*}
These facts show that $\Psi$ is a morphism of double complexes. Since
\begin{align*}
\ddot{B} \xcirc \Psi_{vw}(\bx,\byy) & = \ddot{B}\left( \frac{1}{w+1} \left(\ov{N}(\bx),\Upsilon\xcirc \ov{N}(\bx)\right) + (0,\ov{N}(\byy)) \right)\\
& = (0,\ov{N}(\bx))= \Psi_{v,w+1} (0,\bx) = \Psi_{v,w+1} \xcirc\wt{B}(\bx,\byy),
\end{align*}
$\Psi$ is a morphism of mixed complexes. To finish the proof it suffices to note that $\Psi$ is bijective.\qed

\medskip

\noindent{\bf Proof of Proposition~\ref{pr4.9}.}\enspace By equality~\eqref{e1} and the definition of $\Psi$ (see Theorem~\ref{th4.7}), we have
$$
P(0,\byy) = \frac{1}{w} (0,N(y)) = \frac{1}{w} \Psi(0,[\byy]),
$$
which proves equality~\eqref{et1}. The case $w=0$ of equality~\eqref{et2} follows from the fact that, by equality~\eqref{e2},
$$
P(\bx,0) = (\bx,0) = \Psi([\bx],0).
$$
Assume now that $w>0$. Then, by Corollary~\ref{cor4.4},
$$
P(\bx,0) = \frac{1}{w+1}(N(\bx),0) + \sum_{j=0}^{w-1}\sum_{i=0}^w \frac{2j+2i-2w+1}{2w(w+1)} \bigl(0,t^j\xcirc \varrho_n\xcirc t^i(\bx)\bigr)
$$
and by the definitions of $\Upsilon$, $\sigma'$ and Lemma~\ref{leB.4},
\begin{align*}
\Upsilon\xcirc \ov{N}([\bx]) &= \frac{(w-1)(w+1)}{2w} d'\xcirc\ov{N}([\bx])- \sigma'\xcirc d\xcirc \ov{N}([\bx])\\
&= \frac{(w-1)(w+1)}{2w} d'\xcirc\ov{N}([\bx])- \sigma'\xcirc d'\xcirc \ov{N}([\bx]) - \sigma'\xcirc \varrho_n\xcirc \ov{N}([\bx])\\
& = \frac{(w-1)(w+1)}{2w} \ov{N}\xcirc \ov{d}([\bx])- \sigma'\xcirc\ov{N}\xcirc \ov{d}([\bx])) - \sigma'\xcirc \varrho_n\xcirc \ov{N}([\bx])\\
& = \frac{w-1}{2w} \ov{N}\xcirc \ov{d}([\bx]) - \sum_{j=0}^{w-1}\frac{w-1-j}{w} t^j\xcirc \varrho_n\xcirc \ov{N}([\bx])\\
& = \frac{w-1}{2w} N\xcirc d(\bx) - \sum_{j=0}^{w-1}\sum_{i=0}^w\frac{w-1-j}{w} t^j\xcirc \varrho_n\xcirc t^i(\bx).
\end{align*}
Hence,
\begin{align*}
P(\bx,0)-\Psi([\bx],0) & = P(\bx,0)-\frac{1}{w+1}\left(\ov{N}([\bx]),\Upsilon\xcirc \ov{N}([\bx])\right)\\
&= \sum_{j=0}^{w-1}\sum_{i=0}^w \frac{2i-1}{2w(w+1)} \bigl(0,t^j\xcirc \varrho_n\xcirc t^i(\bx)\bigr) - \frac{w-1}{2w(w+1)} \bigl(0, N\xcirc d(\bx)\bigr)\\
&= \sum_{i=0}^w \frac{2i-1}{2w(w+1)} \bigl(0,N\xcirc \varrho_n\xcirc t^i(\bx)\bigr) - \sum_{i=0}^w \frac{w-1}{2w(w+1)} \bigl(0, N\xcirc \varrho_n\xcirc t^i (\bx)\bigr)\\
&= \sum_{i=0}^w \frac{2i-w}{2w(w+1)} \bigl(0,N\xcirc \varrho_n\xcirc t^i(\bx)\bigr)\\
& = \sum_{i=0}^w \frac{2i-w}{2w(w+1)} \Psi\bigl(0,[\varrho_n\xcirc t^i(\bx)]\bigr),
\end{align*}
as desired.\qed

\end{document}